\documentclass[12pt]{amsart}
\usepackage{amsmath}
\usepackage{amssymb}
\usepackage{amsfonts} 
\usepackage{mathrsfs}
\usepackage{mathtools}

   \newcommand{\Cdb}{\mbox{$\mathbb{C}$}}

   \newcommand{\Ndb}{\mbox{$\mathbb{N}$}}

   \newcommand{\Rdb}{\mbox{$\mathbb{R}$}}
   
   \newcommand{\Tdb}{\mbox{$\mathbb{T}$}}

   \newcommand{\D}{\mbox{${\mathcal D}$}}

   \renewcommand{\H}{\mbox{${\mathcal H}$}}

    \newcommand{\M}{\mbox{${\mathcal M}$}}

\newcommand{\norm}[1]{\Vert#1\Vert}
\newcommand{\bignorm}[1]{\bigl\Vert#1\bigr\Vert}
\newcommand{\Bignorm}[1]{\Bigl\Vert#1\Bigr\Vert}

\newtheorem{theorem}{Theorem}[section]
\newtheorem{lemma}[theorem]{Lemma}
\newtheorem{corollary}[theorem]{Corollary}
\newtheorem{proposition}[theorem]{Proposition}
\newtheorem{definition}[theorem]{Definition}
\theoremstyle{remark}
\newtheorem{remark}[theorem]{\bf Remark}
\theoremstyle{definition}

\numberwithin{equation}{section}

\author[C. Arhancet]{C\'edric Arhancet}
\email{cedric.arhancet@protonmail.com}
\address{6 rue Didier Daurat, 81000 Albi, France}
\author[C. Kriegler]{Christoph Kriegler}
\email{christoph.kriegler@uca.fr}
\address{Universit\'e Clermont Auvergne, 
CNRS, LMBP, F-63000 Clermont-Ferrand, France}
\author[C. Le Merdy]{Christian Le Merdy}
\email{clemerdy@univ-fcomte.fr}
\address{Laboratoire de Math\'ematiques de Besan\c con, UMR 6623,
CNRS, Universit\'e Bourgogne Franche-Comt\'e,
25030 Besan\c{c}on Cedex, France}
\author[S. Zadeh]{Safoura Zadeh}
\email{jsafoora@gmail.com}
\address{Universit\'e Clermont Auvergne, CNRS, LMBP, F-63000 Clermont-Ferrand, France}
\begin{document}

\title{Separating Fourier and Schur multipliers}

\maketitle

\begin{abstract} Let $G$ be a locally compact unimodular group, let $1\leq p<\infty$,
let $\phi\in L^\infty(G)$ and assume that the Fourier multiplier $M_\phi$
associated with $\phi$ is bounded on the noncommutative $L^p$-space $L^p(VN(G))$.
Then $M_\phi\colon L^p(VN(G))\to L^p(VN(G))$ is separating (that is,
$\{a^*b=ab^*=0\}\Rightarrow\{M_\phi(a)^* M_\phi(b)=M_\phi(a)M_\phi(b)^*=0\}$
for any $a,b\in L^p(VN(G))$) if and only if  there
exists $c\in\Cdb$ and a continuous
character $\psi\colon G\to\Cdb$ such that $\phi=c\psi$ 
locally almost everywhere. This provides a characterization of isometric
Fourier multipliers on $L^p(VN(G))$, when $p\not=2$. Next, let 
$\Omega$ be a $\sigma$-finite measure space, let $\phi\in L^\infty(\Omega^2)$
and assume that the Schur multiplier 
associated with $\phi$ is bounded on the 
Schatten space $S^p(L^2(\Omega))$. We prove that this 
multiplier is separating 
if and only if 
there exist a constant $c\in\Cdb$ and two unitaries 
$\alpha,\beta\in L^\infty(\Omega)$ such that 
$\phi(s,t) =c\, \alpha(s)\beta(t)$ a.e. on 
$\Omega^2.$ This provides a
characterization of isometric Schur multipliers
on $S^p(L^2(\Omega))$, when $p\not=2$.
\end{abstract}

\vskip 0.8cm
\noindent
{\it 2020 Mathematics Subject Classification:} Primary 46L51, secondary	43A15, 46B04.

\smallskip
\noindent
{\it Key words:} Fourier multipliers, Schur multipliers, 
noncommutative $L^p$-spaces, isometries.

\bigskip
\section{Introduction} 
Let $\Gamma$ be a locally compact abelian group,  let 
$1\leq p\not=2<\infty$ and let $T\colon L^p(\Gamma)\to L^p(\Gamma)$ 
be a bounded Fourier multiplier.
A classical theorem going back to
Parrott \cite{Par} and Strichartz \cite{St} asserts that
$T$ is an isometry if and only there exists 
$c\in\Cdb$, with $\vert c\vert=1$, and $u\in \Gamma$
such that $T=c\tau_u$. Here $\tau_u\colon L^p(\Gamma)\to L^p(\Gamma)$ 
is the translation operator defined by $\tau_u(f)
=f(\cdotp - u)$.

In the last decade, Fourier multipliers on noncommutative 
$L^p$-spaces associated with group von Neumann algebras emerged as
a major topic in noncommutative analysis, with applications
to approximation properties of operator algebras, to singular 
integrals and Calderon-Zygmund operators, as well
as to noncommutative probability and quantum information. See in particular
\cite{Arh23, GJP, JMP1, JMP2, LD, MR, PRD}. It therefore 
became a natural issue to understand the structure 
of isometric Fourier multipliers in the noncommutative 
framework.
Indeed, the original motivation for this work was to extend the 
Parrott-Strichartz theorem to this setting. 

To be more specific, let $G$ be a locally compact group, let 
$VN(G)$ denote its group von Neumann algebra and
let $\lambda\colon L^1(G)\to VN(G)$ be the contractive
representation associated with the left regular representation of $G$.
Assume that 
$G$ is unimodular. This ensures that the Plancherel 
weight $\tau_G$ on $VN(G)$ is actually a normal semifinite faithful trace.
For any $1\leq p<\infty$, let $L^p(VN(G))$ be the noncommutative 
$L^p$-space associated with $(VN(G),\tau_G)$. A Fourier multiplier
$T\colon L^p(VN(G))\to L^p(VN(G))$ is an operator of the form
$$
T(\lambda(f))=\lambda(\phi f),
$$
where $\phi$ is a fixed element of $L^\infty(G)$ and $f$ 
lies in a suitable
dense subspace of $L^1(G)$. We set $T=T_\phi$
in this case. See the beginning of Section 3 for more details.

We generalize the Parrott-Strichartz theorem by showing the 
following result, in which $VN(G)$ plays the role of 
$\Gamma$ and $\Tdb=\{z\in\Cdb\, :\,\vert z\vert=1\}$:
If $p\not=2$, a Fourier 
multiplier $T_\phi\colon L^p(VN(G))\to L^p(VN(G))$ is
an isometry if and only if there
exists $c\in\Tdb$ and a continuous
character $\psi\colon G\to\Tdb$ such that $\phi=c\psi$ 
locally almost everywhere.

We actually consider the more general class of separating 
Fourier multipliers. Following \cite{LZ}, we say that an 
operator 
$T\colon L^p(M)\to L^p(M)$ acting on some
noncommutative $L^p$-space $L^p(M)$ is separating if
for any disjoint $a,b \in L^p(M)$ (that is, $a^*b=ab^*=0$), 
the images $T(a),T(b)$ are disjoint as well. It is 
well-known that if $p\not=2$, any isometry on
$L^p(M)\to L^p(M)$ is separating. This follows from
Yeadon's characterization of isometries on 
noncommutative $L^p$ spaces \cite{Yeadon} (see also \cite{LZ}).
We prove that for any $1\leq p<\infty$ (including the
case $p=2$), a Fourier 
multiplier $T_\phi$ on $L^p(VN(G))$ is
separating if and only if there
exists $c\in\Cdb$ and a continuous
character $\psi\colon G\to\Tdb$ such that $\phi=c\psi$ 
locally almost everywhere. 

The above two characterizations theorems are established in Section 3.
Section 4 provides complements on 
Fourier multipliers.

Section 5 is devoted to Schur multipliers acting on 
Schatten classes. Let $(\Omega,\mu)$ be a $\sigma$-finite
measure space, let $\phi\in L^\infty(\Omega^2)$
and let $T_\phi$ denote the associated Schur multiplier
acting on the Hilbert-Schmidt space $S^2(L^2(\Omega))$
(see below for details). Let $1\leq p<\infty$ and
assume that $T_\phi$ is bounded on the Schatten space
$S^p(L^2(\Omega))$. We show that 
$T_\phi\colon S^p(L^2(\Omega))\to S^p(L^2(\Omega))$
is separating if and only if 
there exist a constant $c\in\Cdb$ and two unitaries 
$\alpha,\beta\in L^\infty(\Omega)$ such that 
$\phi(s,t) =c\, \alpha(s)\beta(t)$ a.e. on 
$\Omega^2.$ In the case when $p\not=2$, this provides a
characterization of isometric Schur multipliers
on $S^p(L^2(\Omega))$.

\medskip
\section{Preliminaries on separating maps}
Let $\M$ be a semifinite von Neumann algebra equipped 
with a normal semifinite faithful trace $\tau$. Assume that $\M\subset B(H)$ acts 
on some Hilbert space $H$. Let $L^0(\M)$ 
denote the $*$-algebra of all closed, densely defined (possibly unbounded)
operators on $H$, which are
$\tau$-measurable. 
%Let $\widetilde{\tau}$ denote the classical 
%extension of $\tau$ to the positive cone of $L^0(\M)$.
For any $1\leq p<\infty$, 
the noncommutative $L^p$-space $L^p(\mathcal{M})$, 
associated with $(\mathcal{M},\tau)$, can be defined as
$$
L^p(\mathcal{M}):=\bigl\{x\in L^0(\mathcal{M})
\,:\,\tau(\lvert x\rvert^p)<\infty\bigr\}.
$$
Let 
$\|x\|_p:=\tau(\left\lvert x\rvert^p\right)^{\frac{1}{p}}$
for any 
$x\in L^p(\M)$. Then $L^p(\M)$ equipped with $\norm{\,\cdotp}_p$
is a Banach space. We let $L^\infty(\M):=\M$ for convenience and 
we let $\norm{.}_\infty$ denote the operator norm. 
For any $1\leq p<\infty$, let $p'=\frac{p}{p-1}$ 
be the conjugate index of $p$. For any $x\in L^p(\M)$ and $y\in L^{p'}(\M)$, the 
product $xy$ belongs to $L^1(\M)$ and $\vert \tau(xy)
\vert \leq  \norm{x}_p\norm{y}_{p'}$. We further have an isometric
identification $L^{p}(\M)^*\simeq
L^{p'}(\M)$ for the duality pairing given by
$$
\langle y,x\rangle = \tau(xy),\qquad 
x\in L^p(\M),\ y\in L^{p'}(\M).
$$
We let $L^p(\M)^+$ denote the cone of positive elements of $L^p(\M)$. 
The reader is referred to \cite{PX}
and the references therein for details on the algebraic operations
on $L^0(\mathcal{M})$, the construction of $L^p(\M)$, and  for
further properties.

We mention that if $\M=B(H)$ for some Hilbert space $H$, then 
the usual trace ${\rm tr}\colon B(H)^+\to[0,\infty]$ is a 
normal semifinite faithful one and the resulting
noncommutative $L^p$-spaces associated with $(B(H),{\rm tr})$ are the 
Schatten classes $S^p(H)$.

We say that $a,b\in L^0(\M)$ are disjoint if $a^*b=ab^*=0$. 
We say that a bounded operator $T\colon
L^p(\mathcal{M})\to L^p(\M)$, 
$1\leq p\leq \infty$, is separating if whenever 
$a,b\in L^p(\mathcal{M})$ are disjoint then 
$T(a)$ and $T(b)$ are disjoint as well.

A Jordan $*$-homomorphism on a von Neumann algebra $\M$ is a linear 
map $J\colon \M\to\M$ that satisfies $J(a^2)=J(a)^2$ and $J(a^*)=J(a)^*$, 
for every $a\in\M$. It is clear that a Jordan $*$-homomorphism is 
positive, i.e. if $a\in\M^+$ then $J(a)\in\M^+$. We warn the reader
that Jordan $*$-homomorphisms are not always $*$-homomorphisms. For example, the
transposition map on matrices is a Jordan $*$-homomorphism.

However we have the following lemma, in which 
part (1) follows from the identity
$(a+b)^2=a^2 +b^2 + (ab+ba)$ and part (2) is 
given by \cite[10.5.22(iii)]{KR}.

\begin{lemma}\label{commuting}
Let $J:\M\to\M$ be a Jordan $*$-homomorphism. 
\begin{itemize}
\item [(1)] For all $a,b\in\M$, we have $J(ab+ba)=J(a)J(b)
+J(b)J(a)$.
\item [(2)]
If $a,b\in\M$ satisfy $ab=ba$, then we have $J(ab)=J(a)J(b).$
\end{itemize}
\end{lemma}

We also record the following properties for further use.
Here a map $J \colon \M \to \M$ is called normal if it is weak$^*$-continuous.

\begin{lemma}\label{kernel}
Let $J\colon \M\to\M$ be a normal Jordan $*$-homomorphism. 
\begin{itemize}
\item [(1)] The kernel $\ker(J)$ is a $w^*$-closed ideal of $\M$.
\item [(2)] If  $(e_i)_i$
is a bounded
net of $\M$ such that $e_i\to 0$ and $e_i^*\to 0$ 
in the strong operator topology, then
$J(e_i)\to 0$ in the strong operator topology.
\end{itemize}
\end{lemma}

\begin{proof}
(1) Let $J\colon \M\to\M$ be a normal Jordan $*$-homomorphism. 
A well-known  theorem asserts that there exist
two von Neumann algebras $\M_1,\M_2$, a von Neumann algebra 
embedding $\M_1\stackrel{\infty}{\oplus} \M_2\subseteq \M$, a 
normal $*$-homomorphism
$\pi\colon \M\to \M_1$ and a normal anti $*$-homomorphism
$\sigma\colon \M\to \M_2$, such that 
$J(a)=\pi(a)\oplus\sigma(a)$ for all $a\in \M$. 
(See
\cite[Theorem 3.3]{S} or \cite[Corollary 7.4.9.]{HOS} for this result.)
Then $\ker(J)=\ker(\pi)\cap \ker(\sigma)$, hence $\ker(J)$ is 
a weak$^*$-closed ideal.

(2) Let $(e_i)_i$ be a bounded
net of $\M$ such that $e_i\to 0$ and $e_i^*\to 0$ strongly.
Writing $(e_i^*e_i\zeta \vert\eta)=(e_i\zeta
\vert e_i\eta)$, we see that
$(e_i^*e_i\zeta\vert\eta)\to 0$ for all $\zeta,\eta\in\H$. Since 
$(e_i^*e_i)_i$ is bounded, this implies that $e_i^*e_i\to 0$
in the weak$^*$-topology of $\M.$ Consequently, 
$\pi(e_i^*e_i)\to 0$ weakly. Writing 
$\norm{\pi(e_i)\zeta}^2=(\pi(e_i^*e_i)\zeta\vert \zeta)$, we deduce that
$\pi(e_i)\to 0$ strongly. Likewise, using
$e_ie_i^*$ instead of $e_i^*e_i$, we have
that $\sigma(e_i)\to 0$ strongly. Thus, 
$J(e_i)\to 0$ strongly.
\end{proof}

The next statement plays a
fundamental role in the study of separating maps.
It was established independently in \cite{LZ}
and \cite{HRW}.

\begin{proposition}{\cite[Remark 3.3 and Proposition 3.11]{LZ}}\label{Yeadon}
Let $1\leq p<\infty$.
A bounded operator $T\colon L^p(\mathcal{M})\to L^p(\mathcal{M})$
is separating if and only if there exist a normal 
Jordan $*$-homomorphism $J\colon \mathcal{M}\to\mathcal{M}$, a partial isometry $w\in\mathcal{M}$, and a positive operator 
$B$ affiliated with $\mathcal{M}$, which verify the following conditions:
\begin{itemize}
\item [(a)] $T(a)=wBJ(a)$, for all $a\in\mathcal{M}\cap L^p(\mathcal{M})$;
\item [(b)] $w^{\ast}w=J(1)=s(B)$;
\item [(c)] every spectral projection of $B$ commutes with $J(a)$, for all $a\in\mathcal{M}$.
\end{itemize}
Here $s(B)$ denotes the support of $B$. Furthermore, the triple $(w,B,J)$ is unique.
\end{proposition}

It was shown by Yeadon \cite{Yeadon} that all isometries on $L^p(\M)$, 
$p\neq2$, are separating and have the above mentioned factorisation. 
For this reason, for $T$ as above,
we refer to $(w,B,J)$ as the Yeadon triple 
of $T$.

\begin{lemma}\label{2DenseRange} 
Let $T\colon L^p(\M)\to L^p(\M)$
be a separating map and let $(w,B,J)$ denote its 
Yeadon triple. If $T$ has dense
range, then $J(1)=1$ and $w$ is a unitary.
\end{lemma}

\begin{proof}
The proof is an easy modification of \cite[Remark 3.2]{LZ3}.
\end{proof}

\begin{lemma}\label{2Comp}
Let $1\leq p,q<\infty$. Let $T\colon L^p(\M)+L^q(\M)\to L^p(\M)+L^q(\M)$
and assume that $T\colon L^p(\M)\to L^p(\M)$ and $T\colon
L^q(\M)\to L^q(\M)$ are bounded. 
If $T\colon L^p(\M)\to L^p(\M)$ is separating,
then $T\colon L^q(\M)\to L^q(\M)$ is separating as well. 
\end{lemma}

\begin{proof}
Let $\mathcal{E}:=\{e\in \M:e\text{ is a projection with 
$\tau(e)<\infty$}\}$. 
Suppose that $T\colon L^p(\M)\to L^p(\M)$ is separating. 
Since, $\mathcal{E}\subseteq L^p(\mathcal{M})\cap L^{q}(\mathcal{M})$, the operator
$T\colon L^q(\mathcal{M})\to L^q(\mathcal{M})$ also 
preserves disjointness on $\mathcal{E}$. 
By \cite[Remark 3.12 (i)]{LZ}, 
$T\colon L^q(\mathcal{M})\to L^q(\mathcal{M})$ is separating. 
\end{proof}

\section{A characterization of separating Fourier multipliers}
\label{Sec-Separating-Fourier-multipliers}

Let $G$ be a locally compact group with left Haar measure 
$\mu$ defined on the $\sigma$-algebra of Borel sets. 
We will write $ds$ for $d\mu(s)$. Denote by $\lambda$ 
the left regular representation of $G$ defined by
$$
\lambda\colon G\to B(L^2(G));\quad
[\lambda(s)f](t)=f(s^{-1}t).
$$
The left regular representation $\lambda$ determines 
a representation of $L^1(G)$ also denoted by $\lambda$ and defined by 
$$
\lambda\colon L^1(G)\to B(L^2(G)),\quad\lambda(g)\eta=g\ast \eta,
$$
for all $g\in L^1(G)$ and $\eta\in L^2(G)$.
Here, the convolution is $g \ast \eta(t) = \int_G g(s)\eta(s^{-1}t)ds$.
We have that $\lambda(g)=\int_G g(s)\lambda(s)\,
ds$, where the operator integral is understood in the 
strong operator sense.

For any function $g\colon G\to \Cdb$, we let
$$
\check{g}(t)=g(t^{-1})
\qquad\hbox{and}\qquad g^*(t)=\overline{g(t^{-1})},
$$
for all $t\in G$. 

We denote by $e$ the unit element of $G$.
Also for any Borel set $A\subseteq G$, we let 
$\chi_A$ denote the indicator function of $A$.

Let $VN(G)\subseteq B(L^2(G))$ be the von Neumann algebra 
generated by $\{\lambda(s)\,:\,s\in G\}$. This coincides 
with the von Neumann algebra generated by $\{\lambda(g)\, :\,
g\in L^1(G)\}$. When $G$ is abelian, 
we have $VN(G)\simeq L^\infty(\widehat{G})$, 
where $\widehat{G}$ is the dual group of $G$.

We let  $(\,\cdotp\vert\,\cdotp)$ denote the inner product on
$L^2(G)$.
The Fourier algebra of the group $G$ is defined as
$$
A(G)=\{(\lambda(\,\cdotp)\zeta\vert\eta)\, :\,\zeta,\eta\in L^2(G)\}\subseteq C_b(G).
$$
This is a Banach algebra for the pointwise product 
and the norm defined, for any
$\psi\in A(G)$, by
$$
\norm{\psi}_{A(G)} =\inf\{\norm{\zeta}_2\norm{\eta}_2\},
$$
where the infimum runs over all $\zeta,\eta\in L^2(G)$ such that 
$\psi=(\lambda(\,\cdotp)\zeta\vert\eta)$. We note for further use 
that equivalently, we can write
$A(G)=\{\zeta*\eta\,  :\,\zeta,\eta\in L^2(G)\}$.

We recall that $A(G)^*\simeq VN(G)$ isometrically
for the duality pairing given by
\begin{equation}\label{3AG}
\langle\lambda(s),\psi\rangle = \psi(s),
\qquad \psi\in A(G),\ s\in G.
\end{equation}

Assume that $G$ is unimodular. We will use the so-called Plancherel trace 
$$
\tau_G:VN(G)^+\longrightarrow [0,+\infty],
$$
for which we refer to \cite[Section VII.3]{Tak2}
(see also \cite{AK}). We note that $\tau_G$ 
is a normal semifinite faithful trace. This allows to consider 
the noncommutative $L^p$-spaces $L^p(VN(G))$ associated with $\tau_G$. 
We recall that if $G$ is discrete then $G$ is unimodular and $\tau_G$ is 
normalised. Also, if $G$ is abelian then $G$ is 
unimodular and $L^p(VN(G))=L^p(\widehat{G})$, 
where $\widehat{G}$ denotes the dual group of $G$.

It is well-known that for any $g\in L^1(G)\cap L^2(G)$, 
$\lambda(g)\in L^2(VN(G))$ with 
$\|\lambda(g)\|_2=\|g\|_2$ (see 
e.g. \cite[Section 6.1]{AK}). Consequently, 
the restriction of $\lambda$ to $L^1(G)\cap L^2(G)$ 
extends to an isometry from $L^2(G)$ into $L^2(VN(G))$. 
It turns out that
the latter is onto, which yields a unitary 
identification
\begin{equation}\label{2Plancherel}
L^2(VN(G))\simeq L^2(G).
\end{equation}
Using the notation $U_\lambda\colon L^2(G)\to L^2(VN(G))$ for the
above unitary mapping, we have
\begin{equation}\label{3Polar}
\tau_G\bigl(U_\lambda(\zeta)U_\lambda(\eta)\bigr) = \int_G\zeta(t)\check{\eta}(t)\, dt,
\qquad \zeta,\eta\in L^2(G).
\end{equation}
Since $L^1(VN(G))^*\simeq VN(G)$, we have an isometric identification 
$$
A(G)\simeq L^1(VN(G)).
$$
It is not hard to deduce from (\ref{3AG}) and (\ref{3Polar}) that this identification
is given by the mapping $A(G)\to L^1(VN(G))$ taking
$\zeta*\eta$ to $U_\lambda(\check{\eta})U_\lambda(\check{\zeta})$, for all $\zeta,\eta\in L^2(G)$. 
Details are left to the reader.

Let $C_c(G)$ denote the space of continuous and compactly supported functions on $G$. We let $C_c(G)\ast C_c(G)$ denote the linear span of 
$f_1\ast f_2$, where $f_1,f_2\in C_c(G)$. It is well-known that 
$$
\lambda\bigl(C_c(G) \ast C_c(G)\bigr)\subseteq L^1(VN(G))\cap VN(G), 
$$
and that 
$\lambda\bigl(C_c(G) \ast C_c(G)\bigr)$ is 
dense in $L^p(VN(G))$, for all $1\leq p<\infty$.
For a proof, we refer to \cite[Proposition 3.4]{Eym}
for the case $p = 1$, and to \cite[Proposition 4.7]{Daws} for the other cases.
We also note that since $C_c(G) \ast C_c(G)$ is dense in $L^1(G)$, 
$\lambda\bigl(C_c(G) \ast C_c(G)\bigr)$ is 
weak$^*$-dense in $VN(G)$.

\begin{lemma}\label{3L01}
Suppose that $K \subseteq G$ is a compact set. There is a 
function $\psi\in A(G)$ such that $\psi(s)=\mu(K)$, for all $s\in K$. 
\iffalse
Moreover, if $\mu(K) > 0$, there is a function 
$\phi \in C_c(G) \ast C_c(G)$ such that $\phi(s) > 0$ for all $s \in K$.
\fi
\end{lemma}

\begin{proof}
For a proof, we refer to \cite[Proposition 2.3.2.]{KaL18}.
\iffalse
Let $L:=K^{-1}K=\{t^{-1}u:t,u\in K \}$.
Set $\psi:=\chi_K\ast\chi_{L}$.
Note that since $K$ and also $L$ are compact, 
$\chi_K$ and $\chi_L$ belong to $L^2(G)$, so $\psi$ belongs to $A(G)$.
Then for any $s \in K$,
\begin{equation*}
\psi(s)=\int\chi_{K}(t)\chi_{L}(t^{-1}s)dt=\int_K \chi_L(t^{-1}s)dt = \int_K 1 dt =\mu(K).\qedhere
\end{equation*}
\fi

\iffalse
For the second part of the lemma, take again $L = K^{-1} K$ and denote $\psi_K, \psi_L$ the functions obtained in the first part for the compacts $K$ and $L$.
Note that $\psi_K, \psi_L$ are indeed continuous being elements of $A(G)$, and are of compact support.
Put $\phi = \psi_K\ast \psi_L$.
Then for any $s \in K$,
$$
\phi(s) = \int \psi_K(t) \psi_{L}(t^{-1}s) dt > 0
$$
since $\psi_K , \psi_L \geq 0$ and $\psi_K(t) \psi_{L}(t^{-1}s) > 0$ if $t \in K$, and $\mu(K) > 0$.
\fi
\end{proof}

Some of the formulations of the main results in this 
article are easier when the group $G$ is $\sigma$-compact, 
meaning that $G$ is the countable union of compact subsets.
The following well-known lemma relates this property with other 
countability properties that the group $G$ can have. We provide
a proof for the sake of completeness.

\begin{lemma}
\label{lem-secound-countable-sigma-compact}
Let $G$ be a locally compact group.
Then the following implications hold:
\[ G \text{ is second countable }\Longrightarrow G \text{ is }\sigma\text{-compact} \Longleftrightarrow \text{the Haar measure of }G\text{ is }\sigma\text{-finite}.\]
Moreover, the remaining implication is false.
That is, there exists a ($\sigma$-)compact group which is not second countable.
\end{lemma}

\begin{proof}
If $G$ is second countable, then by definition, its topology admits a countable basis.
Since $G$ is locally compact, this basis can be chosen to consist of relatively compact 
sets $O_k,\: k \geq 1$.
Thus,
\[G= \bigcup_{k \in {\mathbb N}} O_k = \bigcup_{k \in
{\mathbb N}} \overline{O_k},\]
so $G$ is $\sigma$-compact.

Next, we show that for the Haar measure $\mu$, $\sigma$-compactness 
and $\sigma$-finiteness are equivalent. Recall  
\cite[Proposition 2.4]{F} that $G$ admits an open, closed, $\sigma$-compact subgroup $G_0$.
Thus, $G = \bigcup_{y \in Y} yG_0$ for some subset 
$Y \subseteq G$ representing the left cosets, where the union is disjoint.
We claim that $Y$ can be chosen at most countable if and only if $G$ is 
$\sigma$-compact, if and only if the Haar measure is $\sigma$-finite.
Indeed, if $Y$ is at most countable, as $yG_0$ is  
$\sigma$-compact for all $y \in Y$, $G$ is $\sigma$-compact.
If $G$ is $\sigma$-compact, say $G = \bigcup_{n \in {\mathbb N}} K_n$ with 
compact $K_n$, then $\mu(K_n) < \infty$ for any $n \in \Ndb$, hence $G$ is $\sigma$-finite.
Finally, suppose that $G$ is $\sigma$-finite, say $G = 
\bigcup_{n \in {\mathbb N}} H_n$ with $\mu(H_n) < \infty$ for any $n \in \Ndb$.
According to \cite[Proposition 2.22]{F}, the fact that 
$\mu(H_n)$ is finite implies that $Y_n := \{y \in Y : \: H_n \cap yG_0 \neq \emptyset\}$ is at most countable.
Thus, $Y' := \{ y \in Y :\: G \cap yG_0 = 
\bigcup_{n \in {\mathbb N}} (H_n \cap yG_0) \neq \emptyset \} = 
\bigcup_{n \in {\mathbb N}} Y_n$ is also at most countable.
But clearly, $Y' = Y$.

Finally, for the last statement, it suffices to take the 
compact non first countable group 
$G = \Tdb^{\mathbb R}$.
\end{proof}

In the sequel we use the space $L^\infty(G)$. 
Its definition requires some care. 
When $G$ is $\sigma$-compact, $L^\infty(G)$ 
is defined in the usual way. But when $G$ 
is not $\sigma$-compact, by the above Lemma 
\ref{lem-secound-countable-sigma-compact}, the left 
Haar measure $\mu$ is not $\sigma$-finite. 
In this case, if we define $L^\infty(G)$ in the  usual way, 
the duality of $L^1(G)$ and $L^\infty(G)$ 
may break down. 
As it is explained in \cite[Section 2.3]{F}, 
it is possible to salvage this duality by modifying 
the definition of $L^\infty(G)$ as follows. 
A set $E\subseteq G$ is called locally Borel 
if $E\cap F$ is Borel whenever $F$ is Borel and $\mu(F)<\infty$. A locally 
Borel set is locally null if $\mu(E\cap F)=0$ 
whenever $F$ is Borel and $\mu(F)<\infty$. 
A function $f\colon G\to\mathbb{C}$ 
is locally measurable if $f^{-1}(A)$ is locally 
Borel for every Borel set $A\subseteq 
\mathbb{C}$. A property is true 
locally almost everywhere if it is true except on a locally null set.

With these definitions in hand, let
$L^\infty(G)$ be the space of all locally measurable functions 
$\phi\colon G\to\Cdb$ that are bounded except 
on a locally null set,
modulo the functions that are zero locally almost everywhere. 
Then $L^\infty(G)$ is a Banach space with the norm
$$
\|\phi\|_\infty=\inf\big\{c:\lvert \phi(t)\rvert\leq 
c\text{ locally almost everywhere}\big\}.
$$

We note that for any $1\leq p<\infty$, any $f\in L^p(G)$ has
a $\sigma$-finite support hence for all
$\phi\in L^\infty(G)$, $\phi f$ is a well-defined 
element of $L^p(G)$.

One may therefore define $\int_G\phi f$ for all $\phi\in L^\infty(G)$
and all $f\in L^1(G)$ and this duality pairing yields an isometric
identification $L^\infty(G)\simeq L^1(G)^*$. 
When $G$ is $\sigma$-compact, $L^\infty(G)$ defined 
as above coincides with the usual one.

\begin{definition}
\label{def-Fourier-multiplier}
For any $\phi\in L^\infty(G)$, let $M_\phi\colon 
\lambda(C_c(G) \ast C_c(G))\to VN(G)$ be defined by 
$$
M_\phi(\lambda(f)):=\lambda(\phi f),
\qquad f\in C_c(G) \ast C_c(G).
$$
For any  $1\leq p < \infty$,
we say that $M_\phi$ is a bounded Fourier multiplier 
on $L^p(VN(G))$ if the above map extends to a bounded operator 
(still denoted by) $M_\phi \colon L^p(VN(G))\to L^p(VN(G))$.
In the sequel we abbreviate this by saying that ``$M_\phi \colon L^p(VN(G))\to L^p(VN(G))$
is a bounded Fourier multiplier".

Likewise, if $p = \infty$, we say that $M_\phi$ 
is a bounded Fourier multiplier on $VN(G)$ if the above map 
extends to a bounded weak$^*$-continuous 
operator $M_\phi \colon VN(G) \to VN(G)$.
\end{definition}

Let $1 \leq p < \infty$.
We recall from \cite[Section 6.1]{AK} 
that if $M_\phi\colon L^p(VN(G))\to L^p(VN(G))$ is 
a bounded Fourier multiplier, then $M_\phi\colon L^{p'}(VN(G))\to 
L^{p'}(VN(G))$ is a bounded Fourier multiplier as well, where 
$p'$ is the conjugate index of $p$.
Moreover, $M_{\check{\phi}}\colon
L^{p'}(VN(G))\to L^{p'}(VN(G))$ is a bounded Fourier multiplier, 
and this operator is actually the adjoint 
of $M_\phi\colon L^p(VN(G))\to L^p(VN(G))$.

Thanks to (\ref{2Plancherel}), we have that 
for any $\phi\in L^\infty(G)$, $M_\phi \colon L^2(VN(G))\to L^2(VN(G))$
is a bounded Fourier multiplier, with 
\begin{equation}\label{Equal}
\norm{M_\phi \colon L^2(VN(G))\longrightarrow 
L^2(VN(G))}=
\norm{\phi}_{L^\infty(G)}.
\end{equation}
Indeed, let us denote as before by $U_\lambda\colon L^2(G)\to L^2(VN(G))$ 
the unitary mapping taking any $f\in L^1(G)\cap L^2(G)$ to
$\lambda(f)$, see 
(\ref{2Plancherel}). Let $\pi\colon L^\infty(G)\to B(L^2(G))$ be
defined by $[\pi(\phi)](f)=\phi f$, for all $\phi\in L^\infty(G)$ and 
all $f\in L^2(G)$.
Then $\pi$ is an isometry. Since
for any $\phi\in L^\infty(G)$,
$U_\lambda\pi(\phi)U_\lambda^*$ coincides with $M_\phi$ on 
$C_c(G) \ast C_c(G)$, we obtain that $M_\phi$ is
a bounded Fourier multiplier on $L^2(VN(G))$ and that
$M_\phi \colon L^2(VN(G))\to L^2(VN(G))$ coincides with 
$U_\lambda\pi(\phi)U_\lambda^*$. Since $U_\lambda\pi(\,\cdotp)U_\lambda^*$ 
is an isometry, the equality (\ref{Equal}) follows.
(See also \cite[Lemma 6.5]{AK}.)
It follows from above that a Fourier multiplier 
$M_\phi \colon L^2(VN(G)) \to L^2(VN(G))$ satisfies 
$$
M_\phi(U_\lambda(f)) = U_\lambda(\phi f),\qquad f \in L^2(G).
$$

\begin{remark}\label{Injective}
If $\phi_1,\phi_2\in L^\infty(G)$ are such that 
$M_{\phi_1}$ and $M_{\phi_2}$ coincide on 
$\lambda(C_c(G) \ast C_c(G))$, then by
(\ref{Equal}), the operator 
$M_{\phi_1-\phi_2}=
M_{\phi_1}-M_{\phi_2} \colon
L^2(VN(G))\to L^2(VN(G))$
is equal to $0$, hence $\phi_1=\phi_2$ locally almost everywhere.
\end{remark}

For the rest of this section, we fix 
a net $(f_i)_{i\in I}$ in 
$C_c(G)\ast C_c(G)$ such that $f_i\geq0$
and $\int_G f_i(s)\, ds=1$ for all $i\in I$,
the supports of $f_i$'s are contained in some compact neighborhood $V_i$ of $e$
where the net $(V_i)_{i \in I},$ is decreasing, and 
$\bigcap_{i\in I}V_i=\{e\}$.
We set $e_i:=\lambda(f_i)$. Then for all $i\in I$, we have
$$
e_i\in L^1(VN(G))\cap VN(G)
\qquad\hbox{and}\qquad
\|e_i\|_{VN(G)}\leq1.
$$

\begin{lemma}\label{ei}
We both have $e_i\to 1$ and $e_i^*\to 1$
in the strong operator topology. 
\end{lemma}

\begin{proof}
Let $\zeta\in L^2(G)$. 
Since each $f_i$ is non-negative and $L^1$-normalized, we have
$$
\|e_i(\zeta)-\zeta\|_{L^2(G)}
=\Big\|\int_G f_i(t)
\left(\lambda(t)\zeta-\zeta\right)\, dt \Big\|_{L^2(G)}
\leq \int_G  f_i(t)\
\|\lambda(t)\zeta-\zeta\|_{_{L^2(G)}}\, dt.
$$
Since $\lambda(t)\to 1$ as $t\to e$ in the strong operator topology,
the assumptions on $(f_i)_{i\in I}$ ensure that
the right hand-side tends to $0$, when $i\to\infty$.  
Hence $\|e_i(\zeta)-\zeta\|_{L^2(G)}\to 0$.

This shows that $e_i\to 1$ strongly.
Since $e_i^*=\lambda(f_i^*)$ and the $f_i^*$
have the same features as the $f_i$, we also have that  
$e_i^*\to 1$ strongly.
\end{proof}

We let $C_b(G)$ be the space of bounded and continuous functions on $G$.

The following lemma shows that the definition
of bounded Fourier multipliers $VN(G)\to VN(G)$ considered
in this paper coincide with the one in \cite{DCH}.

\begin{lemma}\label{Continuity}
Let $\phi\in L^\infty(G)$ and assume that 
$M_\phi$ is a bounded Fourier multiplier on $VN(G)$. Then
there exists $\psi\in C_b(G)$ such that 
$\phi=\psi$ locally almost everywhere, and
$$
M_\psi(\lambda(s))=\psi(s)\lambda(s),\qquad s\in G.
$$
\end{lemma}

\begin{proof}
Assume that
$M_\phi\colon VN(G)\to VN(G)$ is a bounded Fourier multiplier. 
Let us show that for any $s\in G$,
\begin{equation}\label{Pointwise}
M_\phi(\lambda(s))\in {\rm Span}\{\lambda(s)\}.
\end{equation}
Fix $s\in G$. For all $i\in I$, $\lambda(s)e_i=\int_G f_i(t)\lambda(st)dt$, hence 
$M_\phi(\lambda(s)e_i)=\int_G\phi(st)f_i(t)\lambda(st)\,dt$, where these
integrals are defined in the strong operator topology.
Therefore, for all $\varphi\in A(G)$, we have 
\begin{align*}
\langle M_\phi(\lambda(s)e_i),\varphi\rangle&
=\int_G\phi(st)f_i(t)\langle\lambda(st),\varphi\rangle \,dt\\
&=\int_G\phi(st)f_i(t)\varphi(st)\,dt.
\end{align*}
Let $\varphi\in A(G)$ such that $\varphi(s)=0$. 
From the above, we have
\[
\vert\langle M_\phi(\lambda(s)e_i),
\varphi\rangle\vert\leq\|\phi\|_\infty\int_G f_i(t)\vert\varphi(st)-\varphi(s)\vert \,dt.
\]
Since $\varphi$ is continuous, 
the right hand-side of the above inequality tends to zero 
when $i\to\infty$. Moreover by Lemma 
\ref{ei}, $\lambda(s)e_i\to\lambda(s)$ in the 
strong operator topology. Since $(e_i)_{i\in I}$
is bounded, this implies that
$\lambda(s)e_i\to\lambda(s)$ in the weak$^*$-topology. 
Therefore, $M_\phi(\lambda(s)e_i)\to M_\phi(\lambda(s))$ in the weak$^*$-topology. Hence, $\langle M_\phi(\lambda(s)e_i),\varphi\rangle\to\langle M_\phi(\lambda(s)),\varphi\rangle$. We obtain that $\langle M_\phi(\lambda(s)),\varphi\rangle=0$. 
Hence, $M_\phi(\lambda(s))\in\{\varphi\, :\,
\varphi(s)=0\}^\perp$ in the duality 
$A(G)^*\simeq VN(G)$. Since ${\rm Span}\{\lambda(s)\}_\perp=
\{\varphi\, :\, \varphi(s)=0\}$, we deduce (\ref{Pointwise}).

Let $\psi\colon G\to\mathbb{C}$ be the unique 
function such that for all $s\in G$, 
$M_\phi(\lambda(s))=\psi(s)\lambda(s)$. 
The pre-adjoint $A(G)\to A(G)$ of $M_\phi$
is the pointwise multiplication by $\psi$.
According to the comment following
\cite[1.1. Definition]{DCH}, the function $\psi$ is 
therefore continuous. 

Next, we show that $\phi=\psi$ locally almost everywhere. 
For any $f\in C_c(G)\ast C_c(G)$, we have 
$\lambda(f)=\int_G f(t)\lambda(t)\,dt$. This SOT-integral is absolutely
convergent in $VN(G)$, hence 
\begin{align*}
M_\phi(\lambda(f))&=\int_G f(t)M_\phi(\lambda(t)) \,dt\\
&=\int_G f(t)\psi(t)\lambda(t)\, dt\\
& =M_\psi(\lambda(f)).
\end{align*}
This implies that $M_\phi=M_\psi$ on $C_c(G)\ast C_c(G)$,
and the result follows by Remark \ref{Injective}.

For the regularity of $G$ needed in the proof of \cite[1.1. Definition]{DCH}, we refer to \cite[Theorem 2.3.8 p.~53]{KaL18}.
See also the discussion following \cite[Definition 6.3]{AK}.
\end{proof}

\iffalse
Indeed, since $M_\phi$ is by definition weak$^*$-continuous, 
it admits a preadjoint $(M_\phi)_* = M_{\check{\phi}} \colon L^1(VN(G)) \to L^1(VN(G))$.
Note that since $L^1(VN(G))$ is canonically isometric to the Fourier algebra $A(G)$ (see after \eqref{3Polar} hereabove), this entails that the pointwise multiplier $m_{\check{\phi}} : C_c(G) \ast C_c(G) \subseteq A(G) \to A(G), \: f \ast g \mapsto \check{\phi} \cdot f \ast g$ is bounded.
By density, it extends to a pointwise multiplier $A(G) \to A(G), \: f \mapsto \check{\phi} f$.
Since elements of $A(G)$ are continuous, by Lemma \ref{3L01}, we conclude that $\check{\phi}$ coincides on any compact set almost everywhere with a continuous function.
It is not hard to see that this implies that there is a continuous function $\psi \colon G \to \Cdb$ such that $\check{\phi}|_K = \psi|_K$ on any compact $K \subseteq G$.
By Lemma \ref{lem-locally-ae-compact} below, we infer that $\phi$ coincides locally almost everywhere with the continuous function $\check\psi$.
\fi

Let $\Tdb$ denote the unit circle of $\Cdb$.
A homomorphism $\varphi\colon G\to\Tdb$ is called a character.
We let $Hom(G,\mathbb{T})$ denote the collection of all 
characters on $G$. There is a natural isomorphism between 
$Hom(G,\mathbb{T})$ and $Hom(\frac{G}{[G,G]},\mathbb{T})$, 
where $[G,G]$ denotes the commutator subgroup of $G$ 
\cite[(23.8) Theorem p.~358-359]{HR}. 
When $G$ is a perfect group, i.e. $[G,G]=G$, 
the only character on $G$ is the trivial one,
that is, $Hom(G,\mathbb{T})=\{1\}$.
Examples of perfect groups include non-abelian 
simple groups and the special linear groups $SL_n(\mathbb{K})$, 
for a fixed field $\mathbb{K}$.

The following is the main result of this section. In 
view of this theorem and the observation above, we 
see that there are groups with relatively 
few separating Fourier multipliers.

\begin{theorem}\label{3Main} 
Assume that $G$ is a locally compact unimodular 
$\sigma$-compact group,
let $1\leq p<\infty$ and let $\phi\in L^\infty(G)$. 
The following are equivalent.
\begin{itemize}
\item [(i)] The mapping $M_\phi$ 
is a bounded Fourier multiplier on $L^p(VN(G))$,
and the operator 
$M_\phi \colon L^p(VN(G))\to L^p(VN(G))$ is separating.
\item[(ii)] There exist a constant $c\in\Cdb$ and a continuous 
character $\psi\colon G\to\Tdb$ such that $\phi=c\psi$ almost everywhere.
\end{itemize}
\end{theorem}

The proof will be given after a series of intermediate results.

\begin{lemma}\label{weak-star conv}
Let $\M$ be a semifinite von Neumann algebra. 
Let $(x_j)_{j}$ be a net in $\M\cap L^2(\M)$ with 
$\sup\|x_j\|_{\infty}<\infty$. If $\|x_j\|_{2}\to 0$, 
then $x_j\to 0$  in the weak$^*$-topology of $\M$. 
\end{lemma}

\begin{proof}
Take $y\in L^1(\M)\cap L^2(\M)$. We have that 
$$
\vert\langle y,x_j\rangle\vert\leq\|y\|_2\|x_j\|_2\to0,
$$
and therefore $\langle y,x_j\rangle\to 0$. Since $L^1(\M)\cap L^2(\M)$ 
is dense in $L^1(\M)$ and $(x_j)_{j}$ is bounded in $\M$, 
this implies that $\langle y,x_j\rangle\to 0$, for all $y\in L^1(\M)$. That
is, $x_j\to 0$ in the weak$^*$-topology of $\M$.
\end{proof}

In the following lemma, $G$ is not necessarily $\sigma$-compact.

\begin{lemma}\label{3L1} Let $\phi\in L^\infty(G)$. 
Assume that the Fourier multiplier 
$M_\phi \colon L^2(VN(G))\to L^2(VN(G))$ is separating and non-zero. 
\begin{itemize}
\item [(1)]
For any compact $K\subseteq G$, the restriction $\phi\vert_{K}$ is 
non-zero almost
everywhere on $K$.
\item [(2)] The operator $M_\phi \colon L^2(VN(G))\to L^2(VN(G))$
has dense range.
\item [(3)] For any compact $K\subseteq G$, there exists a continuous 
function $\Phi\colon K\to \Cdb$ such that $\phi\vert_K=\Phi$ almost
everywhere on $K$.
\end{itemize}
\end{lemma}

\begin{proof}
(1) Let $K$ be a compact subset of $G$. Set $N_K(\phi):=\{s\in K:\phi(s)=0\}$. 
We show that $N_K(\phi)$ has measure zero. Assume on the contrary 
that $N_K(\phi)$ has positive measure. We show that there exists $a_0\in K$ 
such that for any open neighbourhood $V$ of $a_0$, $\mu(N_K(\phi)\cap V)>0$. 
Assume on the contrary that for any $a\in K$ there is an open neighbourhood of 
$a$, $V_a$, such that $\mu(N_K(\phi)\cap V_a)=0$. Since $K$ is 
compact and $\{V_a\}_{a\in K}$ covers $K$, there is a 
finite subcover $\{V_{a_j}\}_{j=1}^n$ that covers $K$ as well. Now, note that
$$
 \mu(N_K(\phi))=\mu(N_K(\phi)\cap(\cup_{j=1}^n V_{a_j})) =
\mu(\cup_{j=1}^n (N_K(\phi)\cap V_{a_j} ))\leq\sum_{j=1}^n
\mu(N_K(\phi)\cap V_{a_j})=0,
$$
which contradicts the fact that $\mu(N_K(\phi))>0$.

Let $(U_i)_{i\in I}$ be a net of neighbourhoods of $e$,
the unit element of $G$, 
directed by inclusion,
%and all contained in a compact neighbourhood of the identity, 
with $\cap_{i\in I}U_i=\{e\}$. Then 
for all $i\in I$, we
have $\mu(a_0 U_i\cap N_K(\phi))>0$ and we may define 
$$
h_i:=\frac{1}{\mu(a_0 U_i\cap N_K(\phi))}\chi_{U_i}
\chi_{N_K(\phi)}(a_0\,\cdotp).
$$ 
For any $i\in I$, we have that $h_i\geq0$, $\int h_i=1$, and 
$h_i\in L^1(G)\cap L^2(G)$. Moreover, 
$\text{supp}(h_i)=U_i\cap a_0^{-1}N_K(\phi)$ and therefore, 
$\bigcap_{i\in I}\text{supp}(h_i)=\{e\}$ and 
$\phi h_i(a_0^{-1}\,\cdotp)=0$, for all $i\in I$.

Let $(w,B,J)$ be the Yeadon triple of
$M_\phi \colon L^2(VN(G))\to L^2(VN(G))$.
For any $i\in I$, let $\varepsilon_i:=\lambda(h_i(a_0^{-1}\,\cdotp))$. 
We have that $\varepsilon_i\in L^2(VN(G))\cap VN(G)$ and $\|\varepsilon_i\|_{VN(G)}\leq1$. 
By Proposition \ref{Yeadon}(a), $M_\phi(\varepsilon_i)=wBJ(\varepsilon_i)$. 
Also, $M_\phi(\varepsilon_i)=\lambda(\phi h_i(a_0^{-1}\cdotp))=\lambda(0)=0$. 
Therefore, $wBJ(\varepsilon_i)=0$. We now apply Proposition \ref{Yeadon}(b).
Since $w^*w=s(B)$, we have $BJ(\varepsilon_i)=w^*wBJ(\varepsilon_i)$, hence
$BJ(\varepsilon_i)=0$. Further $0\leq J(\varepsilon_i)\leq J(1)=s(B)$, 
hence $J(\varepsilon_i)$ is
valued in $\ker(B)^\perp$. Hence the equality
$BJ(\varepsilon_i)=0$ implies that $J(\varepsilon_i)=0$, that is $\varepsilon_i\in \ker(J)$. 
By Lemma \ref{kernel}(1), this implies
that for any $g\in L^1(G)$ 
and any $i\in I$, we have 
$\lambda(g)\varepsilon_i\in\ker(J)$. 

Now, for any $f\in L^1(G)\cap L^2(G)$, let $g=f(\cdotp a_0)$. We have 
$$
\lambda(g)\varepsilon_i=\lambda(g\ast h_i(a_0^{-1}\,\cdotp))=
\lambda (g(\cdotp a_0^{-1} )\ast h_i)=
\lambda(f\ast h_i).
$$ 
Since $\|f\ast h_i-f\|_2\to 0$,
we have $\|\lambda(f\ast h_i)-
\lambda(f)\|_{L^2(VN(G))}\to 0$, by (\ref{2Plancherel}). 
Note that,
\begin{align*}
\| \lambda(f\ast h_i) \|_\infty
\leq \| f\ast h_i \|_1 
\leq \| f\Vert_1  \Vert h_i \|_1 = \| f\|_1.
\end{align*}
Hence by Lemma \ref{weak-star conv}, 
$\lambda(f\ast h_i)\to\lambda(f)$, 
in the weak$^*$-topology of $VN(G)$. Since
$\ker(J)$ is weak$^*$-closed, by Lemma \ref{kernel}(1), 
we obtain  that
$\lambda(f)$ belongs to $\ker(J)$.
Finally since the space 
$\{\lambda(f) \, :\, f\in L^1(G)\cap L^2(G)\}$ 
is weak$^*$-dense in $VN(G)$,
we deduce  that $\ker(J)=VN(G)$.
Hence $M_\phi=0$, which is a contradiction. 

\smallskip
(2) Consider any $F\in C_c(G)\ast C_c(G)$. This function 
has 
compact support, say $K\subseteq G$.
Set $N_{\delta}=\{s\in K:\vert\phi(s)\vert<\delta\}$, 
for all $\delta>0$. By part (1), 
$\phi\neq0$ almost everywhere on $K$, 
hence we have $\lim_{\delta\to0}\mu(N_{\delta})=0$. 
For any $\delta>0$, set $\displaystyle{g_\delta:=\phi^{-1}
\chi_{N_\delta^c}F}$. Since $\phi^{-1}
\chi_{_{{N_{\delta}^c}}}\in L^\infty(G)$, we have 
$g_{\delta}\in L^1(G)\cap L^2(G)$. Hence, $\lambda(g_\delta)$ 
is well-defined, $\lambda(g_\delta)\in L^2(VN(G))\cap VN(G)$ 
and 
$$
M_\phi(\lambda(g_\delta))=\int\phi(t)g_{\delta}(t)\lambda(t)dt=\lambda(\chi_{N_\delta^c}F).
$$ 
Therefore, $\lambda(F)-M_{\phi}(\lambda(g_\delta))=
\lambda(\chi_{N_\delta}F)$, and we have that
$$
\|\lambda(F)-M_{\phi}(\lambda(g_\delta))\|_2 =
\|\lambda(\chi_{N_\delta}F)\|_2
=\|\chi_{N_\delta}F\|_2,
$$
by \eqref{2Plancherel}.
Since $\mu(N_{\delta})\to 0$, when $\delta\to0$, 
this implies that 
$\lim_{\delta\to0}\|\lambda(F)-M_{\phi}(\lambda(g_\delta))\|_2=0$. 
Hence, $\lambda(F)\in\overline{{\rm ran}(M_\phi)}$. 
By density of $\lambda(C_c(G)\ast C_c(G))$ in $L^2(VN(G))$, 
this implies that the range of $M_\phi$ is dense in $L^2(VN(G))$.

\smallskip
(3) Let $w^*M_\phi$ be the operator taking any 
$a\in L^2(VN(G))$ to $w^*M_\phi(a)$. According to 
\cite[Remark 4.2]{LZ2}, $w^*M_\phi$ is positive, 
that is, it maps $L^2(VN(G))^+$ into itself. We use a 
modification of the argument in \cite[Lemma 6.10]{AK}. 
For all $g\in C_c(G)$, $\lambda(g^*\ast g)\in L^2(VN(G))^+$.
Hence, $w^*M_\phi(\lambda(g^*\ast g))\in L^2(VN(G))^+$. Consequently, 
$$
\left(M_\phi(\lambda(g^*\ast g))\zeta\vert w(\zeta)\right)\geq 0,
$$
for all $\zeta\in L^2(G)$. The calculation in the proof of 
\cite[Lemma 6.10]{AK} shows that
%\begin{align}
%\label{AK1}
$$
\left(M_\phi(\lambda(g^*\ast g))\zeta\vert w(\zeta)\right)
=\int_G\int_G\overline{g(t)}g(s)\phi(t^{-1}s)
\left(\lambda(t^{-1}s)\zeta\vert w(\zeta)\right)
\, ds dt.
$$
%\end{align}
This implies that for all $\zeta\in L^2(G)$,
the function $s\mapsto\phi(s)\left(\lambda(s)\zeta\vert 
w(\zeta)\right)$ is positive definite in the sense of 
\cite[Definition VII.3.20]{Tak2}.

By polarization, this implies that  for all 
$\zeta,\zeta'\in L^2(G)$,
$s\mapsto\phi(s)\left(\lambda(s)\zeta\vert w(\zeta')\right)$
is a linear combination of positive definite functions.
By part (2) and Lemma \ref{2DenseRange}, $w$ is a unitary.
Hence the above actually shows that for all $\zeta,\eta\in L^2(G)$, 
$s\mapsto\phi(s)\left(\lambda(s)\zeta\vert\eta\right)$ 
is a linear combination of positive  definite functions. 
By the definition of $A(G)$, this means that for all 
$\psi\in A(G)$, $\phi\psi$
is a linear combination of positive definite functions. By 
\cite[Corollary 3.22]{Tak2} and its proof,
this implies that for all $\psi\in A(G)$, 
$\phi\psi$ is locally almost everywhere equal to a continuous function. 

Let $K\subseteq G$ be a compact set with $\mu(K)>0$. 
By Lemma \ref{3L01}, there is $\psi\in A(G)$ such that $\psi\vert_K>0$. 
Now, $\left(\phi\psi\right)\vert_K$ is almost everywhere equal to 
a continuous function. Hence, 
$\phi\vert_K$ is almost everywhere equal to 
a continuous function.
\end{proof}

\begin{remark}\label{3Sigma}
Note that if $G$ is $\sigma$-compact, the above lemma implies that 
if $\phi\in L^\infty(G)$ is such that
$M_\phi \colon L^2(VN(G))\to L^2(VN(G))$ is separating, then
$\phi$ is almost everywhere equal to a continuous function. 
\end{remark}

In the next statement, we use the net $(e_i)_{i\in I}$ defined
before Lemma \ref{ei}.

\begin{lemma}\label{3L2}
Let $\phi\in C_b(G)$ and consider the Fourier multiplier
$M_\phi \colon L^2(VN(G))\to L^2(VN(G))$. We have the following convergences in the strong operator topology of $VN(G)$.
\begin{itemize}
\item [(1)] For all $s\in G$, $M_\phi(\lambda(s)e_i)
\xrightarrow[i\to\infty]{} \phi(s)\lambda(s)$.
%\item [(2)] For all $s\in G$, $M_\phi(e_i\lambda(s)) \xrightarrow[i\to\infty]{} \phi(s)\lambda(s)$.
\item [(2)] For all $s\in G$, $M_\phi(\lambda(s)e_i^2)
\xrightarrow[i\to\infty]{} \phi(s)\lambda(s)$.
\item [(3)] For all $s\in G$, $M_\phi(e_i\lambda(s)e_i)
\xrightarrow[i\to\infty]{} \phi(s)\lambda(s)$.
\end{itemize}
\end{lemma}

\begin{proof}
(1) By Lemma \ref{ei},  
it is enough to show that for all $s\in G$,
$M_{\phi}(\lambda(s)e_i)-\phi(s)\lambda(s)e_i$ 
converges to $0$ in $VN(G)$. 
Using the assumptions on $(f_i)_{i\in I}$,
and the fact that $\phi$ is continuous,
we have that
\begin{align*}
\big\|M_{\phi}(\lambda(s)e_i)-\phi(s)\lambda(s)e_i\big\|_\infty
&=\Big\|\int_G\left(\phi(st)f_i(t)\lambda(st)-\phi(s)
f_i(t)\lambda(st)\right)\,dt\Big\|_\infty\\
&\leq \int_G \lvert\phi(st)-\phi(s) \rvert f_i(t)
\, dt\xrightarrow[i\to\infty]{}0.
\end{align*}
This proves the result.

\smallskip
(2) For all $i\in I$, we have 
$e_i^2=\lambda(f_i \ast f_i)$, $f_i \ast f_i\geq 0$, 
$\int_G (f_i \ast f_i) (s) \,d(s) = 1$, and 
${\rm supp}(f_i \ast f_i) 
\subseteq {\rm supp}(f_i) 
\cdot {\rm supp}(f_i)$. Hence 
using $f_i \ast f_i$ instead of $f_i$,
the convergence in (2) can be
shown exactly in the same way as in (1).

\smallskip
(3) We argue as in (1). Let $s\in G$.
Since $e_i\to 1$ in the strong operator topology and 
$(e_i)_{i\in I}$ is bounded, $e_i\lambda(s)e_i\to\lambda(s)$
in the strong operator topology. Hence 
we only need to show that 
$M_\phi(e_i \lambda(s) e_i) - 
\phi(s) e_i \lambda(s) e_i$ converges 
to $0$ in $VN(G)$.
We have that
\begin{align*}
&\bignorm{M_\phi(e_i \lambda(s) e_i) - 
\phi(s) e_i \lambda(s) e_i}_{\infty} \\
& = \Bignorm{\int_G \int_G \phi(rst) f_i(r)f_i(t) \lambda(rst) 
\, dr dt- \int_G \int_G \phi(s) f_i(r)f_i(t) 
\lambda(rst) \,dr dt}_\infty \\
& \leq \int_G \int_G | \phi(rst) - \phi(s)| f_i(r) 
f_i(t) dr dt \xrightarrow[i \to \infty]{} 0,
\end{align*}
again using the assumptions on $(f_i)_{i\in I}$,
and the fact that $\phi$ is continuous.
\end{proof}

\begin{proof}[Proof of Theorem \ref{3Main}]
\

\noindent
\underline{$(ii)\,\Rightarrow\,(i)$:} 
We may assume that $c=1$
and that $\phi=\psi$, i.e. $\phi$ is a continuous character.
Any character is positive definite and maps $e$ to $1$. Hence 
according to \cite[Proposition 4.2]{DCH} and
\cite[Proposition 3.6]{Pau02},
$M_\phi$ is a bounded Fourier multiplier on $VN(G)$, with
$\norm{M_\phi\colon VN(G) \to VN(G)} = 
\norm{M_\phi(1)}_{\infty} = |\phi(e)|=1$.
According to \cite[Lemma 6.4 and Lemma 6.6]{AK}, 
$M_\phi$ is therefore 
a bounded Fourier multiplier on $L^p(VN(G))$ and
$$
\norm{M_\phi\colon L^p(VN(G)) \longrightarrow L^p(VN(G))} 
\leq \norm{M_\phi\colon VN(G) \longrightarrow  VN(G)}.
$$
Thus, 
\begin{equation}\label{norm1}
\norm{M_\phi\colon L^p(VN(G)) \longrightarrow L^p(VN(G))} \leq 1.
\end{equation} 
Since $\phi^{-1}$ is also a continuous character, 
the above argument shows as well that
\begin{equation}\label{norm2}
\norm{M_{\phi^{-1}}\colon L^p(VN(G)) \longrightarrow L^p(VN(G))} \leq 1. 
\end{equation} 
For any $f\in C_c(G) \ast C_c(G)$, we have
\[ M_{\phi^{-1}} M_{\phi}(\lambda(f)) = 
M_{\phi^{-1}}(\lambda(\phi f)) = \lambda(\phi^{-1} \phi f) = \lambda(f) .\]
Similarly, $M_\phi M_{\phi^{-1}}(\lambda(f)) = \lambda(f)$.
Hence $M_\phi$ and $M_{\phi^{-1}}$ are inverse to each other.
It therefore follows from (\ref{norm1}) and 
(\ref{norm2}) that $M_\phi\colon L^p(VN(G)) \to L^p(VN(G))$
is an isometry.

If $p\not=2$, 
it follows from \cite{Yeadon} (see also \cite{LZ})
that $M_\phi\colon L^p(VN(G)) \to L^p(VN(G))$ is separating.
If $p=2$, consider any $1<q\not=2<\infty$. The above reasoning
shows that $M_\phi\colon L^q(VN(G)) \to L^q(VN(G))$ is separating.
Applying Lemma \ref{2Comp}, we deduce that the operator 
$M_\phi\colon L^2(VN(G)) \to L^2(VN(G))$ is separating.

\smallskip
\noindent
\underline{$(i)\,\Rightarrow\,(ii)$:} We assume that
$M_{\phi}\colon L^p(VN(G))\to L^p(VN(G))$ is separating. By
Lemma \ref{2Comp}, $M_{\phi}\colon L^2(VN(G))\to L^2(VN(G))$ is separating
as well. 
Let $(w,B,J)$ be its Yeadon triple. We may
assume that $M_{\phi}$ is non-zero. Then
by Lemma \ref{3L1}(2), $M_{\phi}\colon L^2(VN(G))\to L^2(VN(G))$
has dense range. 
It then follows from Lemma \ref{2DenseRange} that $w$ 
is a unitary and $J(1)=1$. 
According to Remark \ref{3Sigma},
$\phi$ is almost everywhere equal to a continuous function. 
Replacing $\phi$ by this function, we may now assume that $\phi\in C_b(G)$.

%Let us show that $\phi(e)\neq0$. 
For any $s\in G$ we have, by Lemma \ref{3L2} 
and Lemma \ref{commuting}(1),
\begin{align*}
\phi(s)\left(\lambda(s)+\lambda(s)\right)
&=\lim_{j\to\infty} M_{\phi}\left(e_j\lambda(s) 
e_j+\lambda(s) e_j^2\right)\\
&=\lim_{j\to \infty}
wBJ\left(e_j \lambda(s) e_j+\lambda(s) e_j^2\right)\\
&=\lim_{j\to\infty} wB\bigl(J(e_j)J(\lambda(s)e_j)
+J(\lambda(s)e_j)J(e_j)\bigr)\\
&=\lim_{j\to\infty} 
\bigl(M_{\phi}(e_j)
J(\lambda(s)e_j)+M_{\phi}(\lambda(s) e_j) J(e_j)\bigr),
%&=\phi(e)J(\lambda(s))+\phi(s) \lambda(s) J(1)\\
%& =\phi(s)J(\lambda(s))+\phi(s)\lambda(s)\bigr),
\end{align*}
where the limit is taken in the strong operator topology.

By Lemma \ref{ei}, $\lambda(s)e_j\to \lambda(s)$ and
$(\lambda(s)e_j)^*\to \lambda(s)^*$ strongly. 
Hence by Lemma \ref{kernel}(2),
$J(\lambda(s)e_j)\to J(\lambda(s))$ 
strongly. By Lemma \ref{3L2}(1),
$M_{\phi}(e_j)\to \phi(e)$ strongly. Since 
$(M_{\phi}(e_j))_{j\in I}$ is bounded, we deduce that 
$M_{\phi}(e_j)
J(\lambda(s)e_j)\to\phi(e) J(\lambda(s))$ strongly.
Similarly, $M_{\phi}(\lambda(s) e_j) J(e_j)\to 
\phi(s) \lambda(s) J(1) =\phi(s) \lambda(s)$
strongly. It therefore follows from the previous calculation
that 
$\phi(s)\left(\lambda(s)+\lambda(s)\right)=
\phi(e)J(\lambda(s))+\phi(s)\lambda(s)$, that is,
$\phi(s)\lambda(s)=\phi(e)J(\lambda(s))$.

Since $\phi$ is non-zero, this implies that $\phi(e)\not=0$.
Set $\psi:=\phi(e)^{-1}\phi$. It follows from the above that
\begin{equation}\label{J}
J(\lambda(s))=\psi(s)\lambda(s),\qquad s\in G.
\end{equation}
We now show that $\psi$ is a character. 
Let $s,t\in G$ and recall that $\lambda(st)=
\lambda(s)\lambda(t)$. On the one hand, we have that $J(\lambda(st))=\psi(st)\lambda(st)$.
On the other hand, we have that
$J(\lambda(s))J(\lambda(t))=\psi(s)\psi(t)\lambda(st)$.

If $st=ts$, then $\lambda(s)\lambda(t)=\lambda(t)\lambda(s)$, 
hence by Lemma \ref{commuting}(2), we have that 
$J(\lambda(st))= J(\lambda(s)\lambda(t))=J(\lambda(s))J(\lambda(t))$. 
Hence, $\psi(st)=\psi(s)\psi(t)$. 
Assume now that $st\neq ts$. 
By Lemma \ref{commuting}(1), we have 
$$
J(\lambda(s)\lambda(t)+\lambda(t)\lambda(s))=J(\lambda(s))J(\lambda(t))+
J(\lambda(t))J(\lambda(s)).
$$ 
Therefore,
$$
\psi(st)\lambda(st)+\psi(ts)\lambda(ts)
=\psi(s)\psi(t)\lambda(st)+\psi(t)\psi(s)\lambda(ts).
$$ 
Since $\lambda(st)$ and $\lambda(ts)$ are linearly independent, 
the above identity implies that $\psi(st)=\psi(s)\psi(t)$. 
This proves that $\psi$ is a character and therefore, 
$\phi=\phi(e)\psi$ is a scalar multiple of a
character, as requested.
\end{proof}

Let us now give 
a variant of Theorem \ref{3Main} in the general 
case when $G$ is not assumed to be $\sigma$-compact
(see also Remark \ref{3Counter}).
We need the following lemma.

\begin{lemma}\label{lem-locally-ae-compact}
Let $h_1,h_2\colon G\to\Cdb$ be two locally measurable functions. 
The functions $h_1$ and $h_2$ are locally almost everywhere equal 
if and only if for any compact set 
$K\subseteq G$, $h_1\vert_K = h_2\vert_K$, almost everywhere.  
\end{lemma}

\begin{proof}
It is enough to show that if $E\subset G$ is locally Borel, 
then $E$ is locally null if (and only if) $E\cap K$ has measure $0$ 
for any compact set $K\subseteq G$. Assume this property. 
By \cite[Proposition 2.4]{F}, $G$ has an open, 
closed and $\sigma$-compact subgroup, $G_0$. 
Let $Y$ be a subset of $G$ that contains exactly one element 
of each of the left cosets of $G_0$. Set $E_y := E\cap yG_0$ 
for any $y\in Y$. Since $G_0$ is $\sigma$-compact, $yG_0$ is $\sigma$-compact as well.  It then follows that $\mu(E_y)=0$, for all $y\in Y$. Recall from the end of \cite[Section 2.3]{F} 
that $E$ is locally null if and only if $\mu(E_y)=0$, 
for every $y\in Y$. Hence, $E$ is locally null.
\end{proof}

\iffalse

\begin{lemma}\label{lem-Fm-determine-symbol-locally-ae}
Let $1 \leq p \leq \infty$ and $\phi,\psi \in L^\infty(G)$ 
such that $M_\phi,M_\psi \colon L^p(VN(G)) \to L^p(VN(G))$ are bounded Fourier multipliers.
Then $M_\phi = M_\psi$ if and only if $\phi$ and $\psi$ agree locally almost everywhere.
\end{lemma}

\begin{proof}
By density, we have $M_\phi = M_\psi$ if and only if 
$\lambda(\phi f) = \lambda(\psi f)$ for any $f \in C_c(G) \ast C_c(G)$ (case $p <\infty$) resp. 
for any $f \in C_c(G)$ (case $p = \infty$) (see also the beginning of the section), 
if and only if $\phi f$ and $\psi f$ agree as elements of $L^1(G)$ for all such $f$ (as the mapping $\lambda : L^1(G) \to VN(G)$ is injective), if and only if $\phi f = \psi f$ almost everywhere for all such $f$.
If this is true, then according to the second part of Lemma \ref{3L01}, $\phi$ and $\psi$ agree almost everywhere on any compact $K$.
Then according to Lemma \ref{lem-locally-ae-compact}, $\phi$ and $\psi$ agree locally almost everywhere.

On the other hand, if $\phi$ and $\psi$ agree locally almost everywhere and $f$ belongs to $C_c(G)$, then with the compact $K = \text{supp }f$, we have
$$
\{s \in G \colon \phi(s) f(s) \neq \psi(s) f(s) \}\subseteq \{s \in G \colon \phi(s) \neq \psi(s) \} \cap K,$$
which is of measure $0$.
Thus, $\phi f = \psi f$ almost everywhere, and therefore, by the above, $M_\phi = M_\psi$.
\end{proof}

\fi 

\begin{corollary}\label{3Non}
Let $G$ be locally compact unimodular.
Let $1\leq p<\infty$ and let $\phi\in L^\infty(G)$. The following are equivalent.
\begin{itemize}
\item [(i)] The mapping $M_\phi$ is a bounded Fourier multiplier on $L^p(VN(G))$, and the operator $M_\phi \colon L^p(VN(G))\to L^p(VN(G))$ is separating.
\item[(ii)] There exist a constant $c_0\in\Cdb$ and a continuous character $\psi\colon G\to\Tdb$ such that $\phi=c_0\psi$ locally almost everywhere.
\end{itemize}
\end{corollary}

\begin{proof} The proof of the implication 
``$(ii)\,\Rightarrow\,(i)$" in Theorem \ref{3Main}
applies to the non $\sigma$-compact case, so we only
need to prove that (i) implies (ii).

Assume that
$M_{\phi}\colon L^p(VN(G))\to L^p(VN(G))$ is separating.
As in the proof of Theorem \ref{3Main}, we may assume that $p=2$
and that $M_\phi$ is non-zero.
We claim that
\begin{equation}\label{claim}
\exists\varphi\in C_b(G)\ \text{such that}\ \varphi=\phi, 
\text{ locally almost everywhere.}
\end{equation}
To prove this, first note that we may assume that the net $(f_i)_{i\in I}$ 
defined prior to Lemma \ref{ei} has the following property: 
there exists a compact neighbourhood $K_0$ of the unit 
$e$ such that for all $i$, ${\rm supp}(f_i)\subseteq K_0$. 
Let $L\subseteq G$ be compact. Let 
$$
K=K_0LK_0K_0=\{strq:(s,t,r,q)\in K_0\times L\times K_0\times K_0\}.
$$
This is a compact set hence by Lemma \ref{3L1}(3) 
there is a continuous function $\Phi:K\to\mathbb{C}$ 
such that $\Phi=\phi\vert_K$ almost everywhere. 
The proof of Lemma \ref{3L2} shows that for all $s\in L$, 
we have the following convergences in the strong operator topology:
\begin{equation}\label{conv3L8}
M_\phi(\lambda(s)e_i)\to\Phi(s)\lambda(s), \ M_\phi(\lambda(s)e_i^2)\to\Phi(s)\lambda(s),
\quad\hbox{and}\quad
M_\phi(e_i\lambda(s)e_i)\to \Phi(s)\lambda(s).
\end{equation}
In particular, (take $L=\{e\}$) we obtain the existence of $c_0\in\mathbb{C}$ such that 
$$
M_\phi(e_i)\to c_0,$$
in the strong operator topology.

Let $(w,B,J)$ be the Yeadon triple of $M_\phi$. The argument 
in the proof of Theorem \ref{3Main} and the convergence 
properties \eqref{conv3L8} show that for any $L,K,\Phi$ as above we have 
\begin{equation}\label{identity}
c_0J(\lambda(s))=\Phi(s)\lambda(s),\quad\text{for all }s\in L.
\end{equation}
By Lemma \ref{3L1}(1), this implies that $c_0\neq0$.

It follows from the above that for all $s\in G$, $J(\lambda(s))$ 
is proportional to $\lambda(s)$. We therefore have a necessarily unique 
$$
F\colon G\to\mathbb{C};\quad J(\lambda(s))=F(s)\lambda(s), \quad\text{for all }s\in G.
$$
Moreover,  for any $L,K,\Phi$ as above, we have
$$
F\vert_L = c_0^{-1}\Phi\vert_L.
$$
This implies that $F$ is continuous. To prove this, 
fix $s_0\in G$ and apply the above with a compact neighbourhood $L$ of $s_0$. 
Then the continuity of $\Phi\colon K\to\Cdb$
implies the continuity of $F\vert_L$, and hence the continuity of $F$ at $s_0$.

Set $\varphi:=c_0 F$. Again for $L,K,\Phi$ as above we obtain 
that $\varphi\vert_K=\phi\vert_K$, almost everywhere. By Lemma \ref{3L1}, 
this implies that $\varphi=\phi$, locally almost everywhere.
Hence \eqref{claim} is proved.

Since $M_\varphi=M_\phi$, the argument at the end of 
the proof of Theorem \ref{3Main} shows that $\psi:=
c_0^{-1}\varphi$ is a character. 
\end{proof}

\begin{remark}\label{OntoIso}
It follows from the proof of the implication ``$(ii)\,\Rightarrow\,(i)$" 
in Theorem \ref{3Main} that for any continuous character $\psi\colon G\to\Tdb$ and
any $1\leq p<\infty$, the Fourier multiplier
$M_\psi\colon L^p(VN(G))\to L^p(VN(G))$ is an onto isometry.
It therefore follows from Corollary \ref{3Non} that
if $\phi\in L^\infty(G)\setminus\{0\}$ is such that
$M_\phi\colon L^p(VN(G))\to L^p(VN(G))$ is bounded and separating,
then $\norm{M_\phi}^{-1}M_\phi$ is an onto isometry.
\end{remark}

\begin{corollary}\label{3Iso}
Let $1\leq p\not=2 <\infty$ and let 
$\phi\in L^\infty(G)$. The following are equivalent.
\begin{itemize}
\item [(i)] The mapping $M_\phi$ is a bounded Fourier multiplier on $L^p(VN(G))$, and the operator $M_\phi \colon L^p(VN(G))\to L^p(VN(G))$ is an isometry.
\item[(ii)] There exists $\delta\in \Tdb$
such that $\delta\phi$ is locally almost everywhere 
equal to a continuous character.
\end{itemize}
\end{corollary}

\begin{proof}
It follows from the proof of the implication
``$(ii)\,\Rightarrow\,(i)$" in Theorem \ref{3Main}
that for any continuous character
$\psi\colon G\to \Tdb$, $M_\psi \colon L^p(VN(G))\to L^p(VN(G))$ is an isometry. Thus,  (ii) implies (i). Conversely,
assume (i).
Since $p\not=2$, any isometry on $L^p(VN(G))$ is separating, by 
\cite{Yeadon} (see also \cite{LZ}). Hence by Corollary
\ref{3Non}, there exist $c\in\Cdb$
and  a continuous character $\psi\colon G\to\Tdb$ such 
that $\phi=c\psi$ locally almost everywhere. Then $M_\phi=c
M_\psi$, hence $\norm{M_\phi}=\vert c\vert \norm{M_\psi}$,
hence $\vert c\vert =1$. This yields (ii), with $\delta=c^{-1}$.
\end{proof}

Note that Corollary \ref{3Iso} is not true in the case $p=2$. 
Indeed, let $\phi\in L^\infty(G)$. It
follows from the discussion following (\ref{2Plancherel})
that $M_\phi\colon L^2(VN(G))\to L^2(VN(G))$ is an isometry 
if and only if $\vert\phi\vert=1$ 
locally almost everywhere. Yet in general, plenty of 
these isometric Fourier multipliers are not separating.
See Section 4 for more on this.

\begin{remark}\label{3Counter}
Corollary \ref{3Non} may be wrong if 
one replaces ``locally almost everywhere" by
``almost everywhere" in (ii). 
Indeed as in \cite[Section 2.3]{F}, take $G = \Rdb \times \Rdb_{\text{disc}}$, 
where the second factor is 
equipped with the discrete topology. Consider 
$Y = \{0\} \times \Rdb_{\text{disc}}$ which is a closed subset of $G$, 
hence Borel, and set 
$\phi=\chi_Y$. For any compact set $K\subseteq G$, we have
$\mu(K\cap Y)=0$, by \cite[Proposition 2.22]{F}. 
Hence $\phi\vert_K=0$ almost everywhere. By Lemma
\ref{lem-locally-ae-compact}, this implies that
$\phi=0$ locally almost everywhere. Thus
$\phi$ satisfies the properties of 
Corollary \ref{3Non}, with $M_\phi=0$.

However by \cite[Proposition 2.22]{F} again,
$\{s \in G \,:\, \phi(s) \neq 0 \} = Y$ 
has infinite Haar measure, hence $\phi$ 
is not almost everywhere equal to $0$.
Consequently, $\phi$ cannot be almost everywhere equal
to a constant times a continuous character.
\end{remark}

\begin{remark}\label{3Discrete}
In the case when $G$ is discrete, continuity on $G$ is automatic 
and two locally almost everywhere equal functions are equal.
Therefore, in the statement of  Corollary
\ref{3Non}, we can replace part $(ii)$ by 
the following slightly simpler statement: 
there exist $c\in\Cdb$ and a character 
$\psi\colon G\to\Cdb$ such that $\phi=c\psi$.
\end{remark}

\begin{remark}
De Canni\`ere and Haagerup \cite{DCH} defined Fourier multipliers on $VN(G)$, 
including the case when $G$ is not unimodular. 
Let $\phi\in C_b(G)$ and assume that $\phi$ induces 
a Fourier multiplier $M_\phi:VN(G)\to VN(G)$ in the sense of 
\cite[Proposition 1.2]{DCH}. Assume that 
$M_\phi$ is separating. If $G$ is unimodular, 
then $M_\phi:L^2(VN(G))\to L^2(VN(G))$ is separating by \cite[Lemma 3.9]{LZ}. 
Hence, by Corollary \ref{3Non}, $\phi$ is a multiple of a character. 

However, in the general case of a non-unimodular locally compact group, 
the description of separating Fourier multipliers on $VN(G)$ is open.
\end{remark}

\begin{remark}
Let $\Gamma$ be a locally compact abelian group.
Let $G=\widehat{\Gamma}$ be its dual group and recall that 
$L^\infty(\Gamma)=VN(G)$. Let $1\leq p<\infty$.
For any $u\in \Gamma$, let $\tau_u\colon L^p(\Gamma)\to L^p(\Gamma)$ be the 
translation operator defined by $\tau_u(f)
=f(\cdotp - u)$, for all $f\in L^p(\Gamma)$. 
Note that if we regard $u\in \Gamma$ as a character $u\colon
G\to\Tdb$, then
the associated Fourier multiplier 
$M_u\colon L^p(\Gamma)\to L^p(\Gamma)$ coincides with $\tau_u$. 

Let $T\colon L^p(\Gamma)\to L^p(\Gamma)$ be a bounded operator.
Then $T$ commutes with translations, that is, 
$T\circ \tau_u=\tau_u\circ T$ for all $u\in\Gamma$, if and 
only if $T$ is a Fourier multiplier 
(see e.g. \cite[Chapter 4]{Larsen}). Hence 
Corollary  \ref{3Iso} implies the following:
\begin{itemize}
\item [(*)]
{\it If $p\not=2$,
an isometry $T\colon L^p(\Gamma)\to L^p(\Gamma)$ commutes with translations
if and only if there exists $c\in\Tdb$ and $u\in \Gamma$
such that $T=c\tau_u$.}
\end{itemize}
This statement is a classical
result due to Parrott \cite{Par} and Strichartz \cite{St}
and Corollary \ref{3Iso} should be regarded as a generalization of the latter.

We note that the two papers \cite{Par,St} show 
$(*)$ in the case
when $\Gamma$ is not necessarily abelian. If $\Gamma$ is non-abelian,
the statement $(*)$ is not related to 
Corollary \ref{3Iso}.
\end{remark}

\iffalse

We conclude this section with a simple observation.
\begin{corollary}
\label{cor-3Main-isometry}
Let $1 \leq p \not=2 < \infty$ and $
M_\phi:L^p(VN(G))\to L^p(VN(G))$ be an isometric Fourier multiplier.
Then there exists a constant $c\in \Tdb$ and a continuous character $\psi : G \to \Tdb$ such that $\phi$ coincides locally almost everywhere with $c \psi$.
\end{corollary}

\begin{proof}
According to \cite[Theorem 2]{Yeadon}, $M_\phi$ is separating.
Now combine Corollary \ref{3Non} with Lemma \ref{lem-surjective-value-1} below.
\end{proof}

Note that Corollary \ref{cor-3Main-isometry} is not true in the case $p=2$. Indeed, let $\phi\in L^\infty(G)$. Then $M_\phi: L^2(VN(G))\to L^2(VN(G))$ is an isometry if and only if $\vert\phi\vert=1$ locally almost everywhere. Yet in general, plenty of these isometric Fourier multipliers are not separating.

\fi

%%%%%%%%%%%%%Section 4%%%%%%%%%
\section{Completely positive and completely isometric Fourier multipliers}

In this section, we complement the characterizations of separating and isometric $L^p$-Fourier multipliers from Section \ref{Sec-Separating-Fourier-multipliers} with further information. Throughout this section, we assume that $G$ is a unimodular locally compact group. \\

Let $\M$ be a semifinite von Neumann algebra equipped with a normal semifinite faithful trace $\tau$. For any $n\geq 1$, we equip $M_n(\M)$ with ${\rm tr}_n\otimes \tau$, where  ${\rm tr}_n$ is the usual trace on $M_n$. For any $1\leq p\leq \infty$, the resulting noncommutative $L^p$-space $L^p(M_n(\M))$ can be naturally identified (at the algebraic level) with the space of all  $n\times n$ matrices $[x_{ij}]_{1\leq i,j\leq n}$ with entries $x_{ij}$ belonging to $L^p(\M)$.

Let $T\colon L^p(\M)\to L^p(\M)$ be a bounded
operator. For any $n\geq 1$, let $T_n\colon L^p(M_n(\M))\to 
L^p(M_n(\M))$ be defined by $T_n\bigl([x_{ij}]\bigr)=[T(x_{ij})]$, 
for all $[x_{ij}]_{1\leq i,j\leq n}$ in $L^p(M_n(\M))$. 
Following usual terminology, we say that $T$ is completely 
positive if $T_n$ is positive for all $n\geq 1$.
Likewise, we say that $T$ is a complete contraction if $\norm{T_n}\leq 1$ for all $n\geq 1$ and that $T$ is a complete isometry if $T_n$ is an isometry for all $n\geq 1$.

Let $\psi\colon G\to\Tdb$ be a continuous character. Then $\psi$ is positive definite hence by \cite[Proposition 4.2]{DCH}, the Fourier multiplier $M_\phi\colon VN(G)\to VN(G)$ is completely positive. The proof of the implication ``$(ii)\,\Rightarrow\,(i)$"  in Theorem \ref{3Main} actually shows that $M_\psi\colon L^p(VN(G)) \to  L^p(VN(G))$ is  a complete contraction for all $1\leq p<\infty$, and then that $M_\psi\colon  L^p(VN(G))\to  L^p(VN(G))$ is  a complete isometry for all $1\leq p<\infty$. As a consequence of Corollary \ref{3Non}, we therefore obtain the following complement to Remark \ref{OntoIso}.

\begin{corollary}\label{4CI}
Let $1\leq p<\infty$ and let $\phi\in L^\infty(G)\setminus\{0\}$. Assume that $M_\phi \colon L^p(VN(G))\to L^p(VN(G))$ is bounded and separating. Then $\norm{M_\phi}^{-1}M_\phi$ is a complete isometry.
\end{corollary}

\begin{remark}
Let $1\leq p<\infty$ and let $\phi\in L^\infty(G)\setminus\{0\}$ 
such that $M_\phi \colon L^p(VN(G))\to L^p(VN(G))$ 
is bounded and separating. Let $(w,B,J)$ be the Yeadon triple 
of the latter operator. According to Corollary \ref{4CI} 
and \cite[Theorem 3.2]{JRS}, $J$ is a $*$-homomorphism. 

We can make this statement more precise, as follows.
Applying  Corollary \ref{3Non}, let $c\in\Cdb$ let $\psi\colon G\to\Tdb$ 
be the continuous character such that $\phi=c\psi$ locally almost everywhere. 
Then $J\colon VN(G)\to VN(G)$ is the $L^\infty$-Fourier 
multiplier associated with $\psi$, $c=\norm{M_\phi}$,
$B=\vert c\vert\cdotp 1$ and $w=c\vert c\vert^{-1}\cdotp 1$. 
The easy verification is left to the reader.
\end{remark}

\begin{lemma}
\label{lem-surjective-value-1}
Let $1\leq p<\infty$ and let  $\phi\in L^\infty(G)$. If $M_\phi$ is a bounded Fourier multiplier on $L^p(VN(G))$ and
$M_\phi\colon L^p(VN(G))\to L^p(VN(G))$ is an isometry, 
then $|\phi| = 1$ locally almost everywhere.
\end{lemma}

\begin{proof}
The case $p = 2$ follows from the paragraph preceding Remark 
\ref{3Counter}. Assume that $1 \leq p \neq 2 < \infty$.
By Corollary \ref{3Non}, there exist a constant $c\in\Cdb$ and 
a continuous character $\psi\colon G\to\Tdb$ such that $\phi=c\psi$ 
locally almost everywhere. We noticed before Corollary \ref{4CI} that 
$M_\psi$ is a complete isometry.
Since $c M_\psi = M_\phi$ is also an isometry, we must have that $|c| = 1$. Hence, $|\phi|=|c\psi|=|\psi|=1$ locally almost everywhere.
\end{proof}

We have the following partial converse of Lemma \ref{lem-surjective-value-1}.

\begin{proposition}
\label{prop-value-1-character}
Let $\phi\in L^\infty(G)$ such that $|\phi| = 1$ locally almost everywhere, let $1\leq p<\infty$ and assume that $M_\phi\colon L^p(VN(G))\to L^p(VN(G))$ is a bounded Fourier multiplier. If $M_\phi$ is completely positive, then $\phi$ coincides locally almost everywhere with a continuous character $\psi\colon G \to \Tdb$.
\end{proposition}

\begin{proof}
Since $M_\phi$ is completely positive, it follows from \cite[Proposition 6.11]{AK} that $\phi$ is locally almost everywhere equal to a continuous positive definite function. 
Hence, we may assume that $\phi$ is continuous (and positive definite). By \cite[Proposition VII.3.21]{Tak2}, there exist a unitary representation $\pi\colon G \to B(H)$ on a Hilbert space $H$ and a vector $\xi \in H$ such that
\begin{equation}\label{mod}
\phi(s) 
= \langle \pi(s)\xi,\xi \rangle, \qquad s \in G. 
\end{equation}
Since $\vert\phi\vert=1$, we have $\phi(e)=1$. Hence it 
follows from \eqref{mod} that $1 = \phi(e) = \langle \pi(e) \xi,\xi \rangle = \norm{\xi}_H^2$. Given $s \in G$, applying the Cauchy-Schwarz inequality we obtain
\[ 
1 = |\phi(s)| 
= |\langle \pi(s)\xi,\xi \rangle| 
\leq \norm{\pi(s)\xi}_H \norm{\xi}_H 
= \norm{\xi}_H^2  = 1. 
\]
It follows from the equality condition 
in the  Cauchy-Schwarz inequality that there is $\psi(s) \in \Cdb$ 
such that $\pi(s)\xi = \psi(s)\xi$.

Now, for any $s,t \in G$, on the one hand 
\[ 
\pi(st)\xi = \psi(st)\xi,
\]
and on the other hand,
\[
\pi(s)\pi(t)\xi= \pi(s)\psi(t)\xi=\psi(s) \psi(t)\xi.
\]
Hence, $\psi(st) = \psi(s)\psi(t)$. Finally, 
$\phi(s) = \langle \pi(s)\xi,\xi\rangle= \psi(s) \norm{\xi}_H^2 = \psi(s)$ 
for all $s\in G$. Therefore, $\phi = \psi$  is a 
character.
\end{proof}

When $p = 1$ and $G$ is assumed to be amenable we can
change the assumption of complete positivity 
in Proposition \ref{prop-value-1-character} 
into mere contractivity.

\begin{proposition}
\label{prop-amenable-characterization-L1-isometries}
Let $G$ be an amenable unimodular locally compact group. 
Let $\phi\in L^\infty(G)$ and assume that
$M_\phi\colon L^1(VN(G))\to L^1(VN(G))$ is a contractive Fourier multiplier. 
The following are equivalent.
\begin{itemize}
\item [(i)] $M_\phi$ is an isometry.
\item [(ii)] $|\phi| = 1$ locally almost everywhere.
\item [(iii)] There exist $c\in\Tdb$ and a continuous character $\psi \colon G \to \Tdb$ 
such that $\phi=c\psi$ locally almost everywhere.
\end{itemize}
\end{proposition}

\begin{proof}
The implication ``$(i)\,\Rightarrow\,(ii)$"  
is established in Lemma \ref{lem-surjective-value-1}. 
The implication ``$(iii)\,\Rightarrow\,(i)$"  was already discussed
several times (see, for example, Remark \ref{OntoIso}). 
We now show that ``$(ii)\,\Rightarrow\,(iii)$".  
Since $M_\phi$ is a bounded Fourier multiplier on $L^1(VN(G))$, 
we may assume that $\phi$ is continuous, by Lemma \ref{Continuity}. 
Further since $G$ is amenable, symbols of Fourier multipliers on $VN(G)$ 
coincide with the Fourier-Stieltjes algebra of $G$. This classical result is
mentioned in \cite[p. 456]{DCH}, see also \cite[Theorem 1]{Herz}. 
Hence by \cite[Lemma 2.14]{Eym}, there exist a 
unitary representation $\pi\colon G \to B(H)$ on a Hilbert 
space $H$ and vectors $\xi,\eta$ in $H$ such that
\[ 
\phi(s) = \langle \pi(s) \xi,\eta \rangle,\quad s\in G,
\qquad \text{and}\qquad \|\xi\|_H = \|\eta\|_H=1.
\]
Assume $(ii)$. Multiplying $\phi$ by $\overline{\phi(e)}$, 
we may assume that $\phi(e) = 1$. This implies 
that $1=\langle \pi(e) \xi,\eta \rangle =\langle \xi,\eta \rangle$. 
Since $\|\xi\|_H = \|\eta\|_H=1$, we deduce that $\eta=\xi$. 
Thus, $\phi$ satisfies (\ref{mod}). 
The proof of Proposition \ref{prop-value-1-character} 
therefore shows  that $\phi$ is a character.
\end{proof}

%%%%%%section 5%%%%%%%%%%%%%%

\section{A characterization of separating Schur multipliers}
Let $(\Omega,\Sigma,\mu)$ be a $\sigma$-finite measure space.
For any $f\in L^2(\Omega^2)$, 
let $S_f\colon L^2(\Omega)\to L^2(\Omega)$
be the bounded operator defined by 
$$
[S_f(h)](s) = \int_\Omega f(s,t)h(t)\, dt,
\qquad h\in L^2(\Omega).
$$
We recall that $S_f\in S^2(L^2(\Omega))$ and that the mapping 
$f\mapsto S_f$ is a unitary operator from $L^2(\Omega^2)$ onto $S^2(L^2(\Omega))$,
see e.g. \cite[Theorem VI. 23]{RS}.

Let $\phi\in L^\infty(\Omega^2)$. According to the above 
identification $L^2(\Omega^2)\simeq S^2(L^2(\Omega))$, one may define 
a bounded operator $T_\phi\colon
S^2(L^2(\Omega))\to S^2(L^2(\Omega))$ by 
\begin{equation}\label{4DefSchur}
T_\phi(S_f) = S_{\phi f},\qquad f\in L^2(\Omega^2).
\end{equation}
Moreover the norm of this operator is equal to $\norm{\phi}_\infty$.
The operator $T_\phi$ is called a Schur multiplier.

Let $1\leq p<\infty$. We say that $T_\phi$ is 
a bounded Schur multiplier on the Schatten space 
$S^p(L^2(\Omega))$ if the restriction
of $T_\phi$ to $S^p(L^2(\Omega))\cap S^2(L^2(\Omega))$
extends to a bounded operator from $S^p(L^2(\Omega))$ into itself.
Schur multipliers as defined in this section
go back at least 
to Haagerup \cite{Haa} and Spronk \cite{Spronk}.

For any $\alpha\in L^\infty(\Omega)$, we let
${\rm Mult}_\alpha\in B(L^2(\Omega))$ be
the multiplication operator taking $h$ to $\alpha h$
for all $h\in L^2(\Omega)$. Then we let 
$$
\D(\Omega)=\bigl\{{\rm Mult}_\alpha\, :\, 
\alpha\in L^\infty(\Omega)\bigr\}.
$$
This is von Neumann sub-algebra of $B(L^2(\Omega))$, which is
isomorphic (as a von Neumann algebra) to $L^\infty(\Omega)$. We will 
use the classical fact that
\begin{equation}\label{4Commute}
\D(\Omega)'=\D(\Omega),
\end{equation}
where $\D(\Omega)'$ stands for the commutant of $\D(\Omega)$. In other words,
a bounded operator $V\colon L^2(\Omega)\to L^2(\Omega)$ belongs
to $\D(\Omega)$ if and only if $V\circ {\rm Mult}_\alpha
={\rm Mult}_\alpha\circ V$ for all $\alpha\in L^\infty(\Omega)$.

We note that for any $\alpha\in L^\infty(\Omega)$, 
the mapping $x\mapsto {\rm Mult}_\alpha\circ x$
is a Schur multiplier. Indeed it coincides with
$T_\phi$, where the symbol $\phi\in L^\infty(\Omega^2)$
is given by $\phi(s,t)=\alpha (s)$.
Likewise, for any $\beta\in L^\infty(\Omega)$,
the mapping $x\mapsto x\circ {\rm Mult}_\beta$
is a Schur multiplier, with symbol 
$\phi$ given by $\phi(s,t)=\beta(t)$.

\begin{theorem}\label{4Main}
Let $\phi\in L^\infty(\Omega^2)$ and let $1\leq p<\infty$.
The following are equivalent.
\begin{itemize}
\item [(i)] The mapping $T_\phi$ 
is a bounded Schur multiplier on $S^p(L^2(\Omega))$, and the resulting operator 
$T_\phi \colon S^p(L^2(\Omega))\to S^p(L^2(\Omega))$ is separating.
\item[(ii)] There exist a constant $c\in\Cdb$ and two unitaries 
$\alpha,\beta\in L^\infty(\Omega)$ such that 
$$
\phi(s,t) =c\, \alpha(s)\beta(t)\qquad \hbox{for almost every}\ (s,t)\in\Omega^2.
$$
\item[(iii)] There exist a constant $c\in\Cdb$ and two unitaries 
$\alpha,\beta\in L^\infty(\Omega)$ such that 
\begin{equation}\label{4Factor0}
T_\phi(x)= c\,{\rm Mult}_\alpha\circ x\circ {\rm Mult}_\beta,
\qquad x\in S^2(L^2(\Omega)).
\end{equation}
\end{itemize}
\end{theorem}

\begin{proof}
\

\noindent
\underline{$(ii)\,\Rightarrow\,(iii)$:} Let $c,\alpha,\beta$
as in (ii) and let $x\in S^2(L^2(\Omega))$. Let $f\in L^2(\Omega^2)$ such that
$x=S_f$. Then for all $h\in L^2(\Omega)$, we have
$$
[x\circ {\rm Mult}_\beta (h)](s)
=\int_\Omega f(s,t)\beta(t)h(t)\,d\mu(t),
$$
hence
\begin{align*}
[c\, {\rm Mult}_\alpha\circ  x\circ {\rm Mult}_\beta (h)](s)
& =c\alpha(s) \int_\Omega f(s,t)\beta(t)h(t)\,d\mu(t)\\
&= \int_\Omega \phi(s,t) f(s,t)h(t)\, d\mu(t),
\end{align*}
for a.e. $s\in\Omega$. This shows (\ref{4Factor0}).

\noindent
\underline{$(iii)\,\Rightarrow\,(i)$:} Assume (\ref{4Factor0}) for some unitaries
$\alpha,\beta\in L^\infty(\Omega)$. It is plain that $T_\phi$ extends to 
a bounded operator on $S^p(L^2(\Omega))$ and that the  identity
(\ref{4Factor0}) holds true on $S^p(L^2(\Omega))$.

Let $x,y\in S^p(L^2(\Omega))$ such that 
$x^*y=xy^*=0$. Then 
$$
\bigl({\rm Mult}_\alpha\circ x\circ {\rm Mult}_\beta\bigr)^*
\bigl({\rm Mult}_\alpha\circ y\circ {\rm Mult}_\beta\bigr)
= {\rm Mult}_{\beta}^*\circ  x^* \circ {\rm Mult}_{\alpha}^*{\rm Mult}_{\alpha} 
\circ y\circ  {\rm Mult}_\beta.
$$
Since $\alpha$ is a unitary of $L^\infty(\Omega)$, the operator 
${\rm Mult}_{\alpha}$ is a unitary of $B(L^2(\Omega))$, hence
the right hand-side of the above equality is equal to 
${\rm Mult}_{\beta}^*\circ  x^* 
y\circ  {\rm Mult}_\beta$, hence to $0$. Thus $T_\phi(x)^*T_\phi(y)=0$. Likewise, 
$T_\phi(x) T_\phi(y)^*=0$. This shows that 
$T_\phi \colon S^p(L^2(\Omega))\to S^p(L^2(\Omega))$ is separating.

\smallskip
\noindent
\underline{$(i)\,\Rightarrow\,(ii)$:} For convenience
we let $H=L^2(\Omega)$ throughout this proof.
Owing to Lemma \ref{2Comp}, we may suppose that $p=2$. 
We let $(w,B,J)$ denote the Yeadon triple of the separating map
$T_\phi\colon S^2(H)\to S^2(H)$. 

We may assume that
$T_\phi$ is non-zero. Since $B(H)$ is a factor,
it follows from 
\cite[Lemma 4.3]{LZ3} that
$T_\phi$ is 1-1. Applying the definition 
of $T_\phi$, see (\ref{4DefSchur}), this implies
that $\phi\not=0$ almost everywhere. 
Applying this definition again, we obtain  that 
$T_\phi$ has dense range.
By Lemma \ref{2DenseRange}, we deduce that $w$ is a unitary
and that $J(1)=1$.

Let $w^*T_\phi$ denote the operator on $S^2(H)$
taking any $x\in S^2(H)$ to $w^*T_\phi(x)$.
According to \cite[Theorem 3.3]{S} (see also 
\cite[Corollary 7.4.9.]{HOS}), there exists a projection 
$q\in B(H)$ such that 
$x\mapsto qJ(x)$ is a $*$-homomorphism and
$x\mapsto (1-q)J(x)$ is an anti-$*$-homomorphism.
As explained in \cite[Remark 4.3]{LZ2}, this implies that
$w^*T_\phi$ is valued in 
$$
L^2\bigl(qB(H)q\bigr)\stackrel{2}{\oplus}
L^2\bigl((1-q)B(H)(1-q)\bigr)\,\subset S^2(H).
$$
Since $T_\phi$ has dense range and
$w$ is a unitary, $w^*T_\phi$ has dense range as well.
This forces $q$ to be equal 
either to $0$ or $1$.
Thus $J\colon B(H)\to B(H)$ is either a $*$-homomorphism or 
an anti-$*$-homomorphism.

Assume first that $J$ is a $*$-homomorphism. Recall 
that $J$ is normal. According to the description
of normal $*$-homomorphisms (see e.g. \cite[Theorem IV.5.5]{Tak}), there exist
a Hilbert space $E$ and a unitary $u\colon H
\to  H\stackrel{2}{\otimes} E$ such that
$$
J(a) = u^*(a\otimes I_E)u,\qquad 
a\in B(H).
$$
Here $H\stackrel{2}{\otimes} E$ stands for the Hilbertian tensor
product of $H$ and $E$ and we regard 
$$
B(H)\otimes B(E)\subset 
B\bigl(H\stackrel{2}{\otimes} E\bigr)
$$
in the usual way. For all $x\in S^2(H)$, we have
$T_\phi(x)=wBJ(x)$ hence $w^*T_\phi(x)=BJ(x)$. This implies that
\begin{equation}\label{u}
u(w^*T_\phi(x))u^*= uBu^*(x\otimes I_E),\qquad x\in S^2(H).
\end{equation}
Since $B$ commutes with the range of $J$, the operator
$uBu^*$ commutes with $x\otimes I_E$ for all $x\in S^2(H)$.
Consequently, $uBu^*= I_H\otimes c$ for some positive operator $c$ 
acting on $E$. Then it follows from (\ref{u}) that 
$c\in S^2(E)$ and that 
$$
w^*T_\phi(x) = u^*(x\otimes c)u,\qquad 
x\in S^2(H).
$$
Now recall that $w^*T_\phi$ has dense range. The above identity therefore implies that 
$E=\Cdb$. Thus $c\in\Cdb\setminus\{0\}$, $u$ is a unitary of $B(H)$
and $w^*T_\phi(x) = c\,u^*xu$ for all $x\in S^2(H)$.
Let $v=wu^*$. This is a unitary of $B(H)$ and we obtain that
$$
T_\phi(x) = c\, vxu,\qquad x\in S^2(H).
$$

Let $(\,\cdotp\vert\,\cdotp)$ denote the inner product on $H$.
For any $g,h\in H$, let $g\otimes\overline{h}\in B(H)$ denote the 
rank one operator taking any $\xi\in H$ to $(\xi\vert h)\, g$.
Then $v(g\otimes\overline{h})u = v(g)\otimes \overline{u^*(h)}$.

Schur multipliers commute with each other, hence for any 
$\delta\in \D(\Omega)$, we have
$$
T_\phi\bigl(\delta x\bigr) = 
\delta T_\phi(x), \qquad x\in S^2(H).
$$
Thus $v\delta xu=\delta vxu$ for all $\delta\in \D(\Omega)$
and all $x\in S^2(H)$.
Applying this with $x=g\otimes\overline{h}$ and 
using the identities $\delta(g\otimes\overline{h})=
\delta(g)\otimes\overline{h}$ and $v(g\otimes\overline{h})=
v(g)\otimes\overline{h}$, we
deduce that 
$v(\delta(g)\otimes  \overline{h})u=
\delta (v(g)\otimes\overline{h})u$ and hence
$$
v\delta(g)\otimes \overline{u^*(h)}
=\delta  v(g)\otimes \overline{u^*(h)},\qquad
g,h\in H,\ \delta\in\D(\Omega).
$$
Since $u^*\not=0$, this implies that 
$v\delta=\delta v$ for all $\delta\in\D(\Omega)$. Thus
$v$ commutes with $\D(\Omega)$. According to (\ref{4Commute}), 
this implies that $v\in\D(\Omega)$. Thus there exists a unitary
$\alpha\in L^\infty(\Omega)$ such that $v={\rm Mult}_\alpha$.
Likewise there exists a unitary
$\beta\in L^\infty(\Omega)$ such that $u={\rm Mult}_\beta$.
We therefore obtain the identity (\ref{4Factor0}), from which (ii) follows at once.

Assume now that $J$ is 
an anti $*$-homomorphism. For any
$f\in L^2(\Omega^2)$, let $\widetilde{f} \in L^2(\Omega^2)$
be defined by $\widetilde{f}(s,t)=f(t,s)$, for a.e. 
$(s,t)\in\Omega^2$. Next,  if $x=S_f$, set 
${}^t \! x = S_{\widetilde{f}}$. It is clear that
the mapping $x\mapsto {}^t \! x$
on $S^2(L^2(\Omega))$ 
extends to a normal anti $*$-homomorphism
$$
\rho\colon B(H)\longrightarrow B(H).
$$
This mapping is an analog of the
transposition map on matrices. 
Obviously, the composition map  $J\circ\rho\colon B(H)\to B(H)$ 
is a normal $*$-homomorphism. Now arguing as in the
$*$-homomorphism case, we obtain the existence
of a constant $c\in\Cdb\setminus\{0\}$ and of two unitaries 
$\alpha,\beta\in L^\infty(\Omega)$ such that 
$$
T_\phi(x)= c\,{\rm Mult}_\alpha\circ {}^t \! x\circ {\rm Mult}_\beta,
\qquad x\in S^2(L^2(\Omega)).
$$
Since $\alpha,\beta$ are unitaries, ${\rm Mult}_\alpha$
and ${\rm Mult}_\beta$ are unitaries as well
and we have ${\rm Mult}_\alpha^{-1} = {\rm Mult}_{\overline{\alpha}}$
and ${\rm Mult}_\beta^{-1} = {\rm Mult}_{\overline{\beta}}$.
Writing ${}^t \! x =  c^{-1} {\rm Mult}_{\overline{\alpha}}
\circ T_\phi(x) \circ {\rm Mult}_{\overline{\beta}}$, we therefore
deduce that $x\mapsto {}^t \! x$ is a Schur multiplier.

Let us show that this is impossible,  except if $L^2(\Omega)$ has dimension 1. 
If $x\mapsto {}^t \! x$ is a Schur multiplier,
then there exists $\phi_0\in L^\infty(\Omega^2)$ such that
\begin{equation}\label{4Trans}
\phi_0(s,t) g(s)h(t) =h(s)g(t)\qquad a.e.\hbox{-}(s,t)\in\Omega^2,
\end{equation}
for all $g,h\in L^2(\Omega)$. If $L^2(\Omega)$ has dimension $\geq 2$, 
then there exist $F_1,F_2\in\Sigma$ such that
$0<\mu(F_1)<\infty$, $0<\mu(F_2)<\infty$ and $F_1\cap F_2=\emptyset$.
The indicator functions $g=\chi_{F_1}$ and $h=\chi_{F_2}$ belong
to $L^2(\Omega)$. Applying (\ref{4Trans}) to these functions,   
we obtain that $h(s)g(t)=0$ for almost every $(s,t)\in F_2\times F_1$. 
Since $h(s)g(t)=1$ for $(s,t)\in F_2\times F_1$ and
$$
(\mu\otimes\mu)(F_2\times F_1) =\mu(F_2)\mu(F_1)>0, 
$$
we get a contradiction.

Now if we are in the trivial case when $L^2(\Omega)$ has dimension
1, then (ii) holds true.
\end{proof}

\begin{remark}\label{5Iso}
Let $\phi\in L^\infty(\Omega^2)$, let $1\leq p\not=2<\infty$ and
assume that $T_\phi$ 
is a bounded Schur multiplier on $S^p(L^2(\Omega))$. 
It follows from Theorem \ref{4Main} and \cite{Yeadon} that if $T$ 
is an isometry, then there exist two  unitaries 
$\alpha,\beta\in L^\infty(\Omega)$ such that 
$\phi(s,t)=\alpha(s)\beta(t)$ for a.e. $(s,t)\in\Omega^2$. 
It is clear that the converse is true.
For the discrete case (see the following remark), this has been proved in \cite{Arh11}.
\end{remark}

\begin{remark}\label{4Discrete}
Let $I$ be an index set and let $(e_i)_{i\in I}$ be the standard basis 
of $\ell^2_I$. Any $x\in B(\ell^2_I)$ can be represented by a 
matrix $[x_{ij}]_{i,j\in I}$ defined by 
$x_{ij}=(x(e_j)\vert e_i)$ for all $i,j\in I$. 
Of course any finitely supported matrix 
$[x_{ij}]_{i,j\in I}$ represents an element of $B(H)$
(actually a finite rank one), and we let 
$\norm{[x_{ij}]}_p$ denote the $S^p(\ell^2_I)$-norm of this element.

Let ${\mathfrak m}=\{m_{ij}\}_{(i,j)\in I^2}$ be a bounded 
family of complex numbers. If we apply the definitions
of this section to $\Omega=I$ equipped with the
counting measure, the Schur multiplier
$T_{\mathfrak m}$ is defined on finitely supported matrices by
$$
T_{\mathfrak m}\bigl( [x_{ij}]\bigr)=[m_{ij}x_{ij}].
$$
It therefore follows from
Remark \ref{5Iso} that the following are equivalent:
\begin{itemize}
\item [(i)] There exists $1\leq p\not=2<\infty$ such that
$$
\norm{[m_{ij}x_{ij}]}_p=\norm{[x_{ij}]}_p
$$
for all finitely supported matrices $[x_{ij}]_{i,j\in I}$. 

\smallskip
\item [(ii)] There exist two families $(\alpha_i)_{i\in I}$ and 
$(\beta_j)_{j\in I}$ in $\Tdb$ such that
$$
m_{ij}=\alpha_i\beta_j,\qquad\hbox{for all} \ 
(i,j)\in I^2.
$$
\end{itemize}
\end{remark}

We conclude with a characterisation of a particular class of Schur multipliers, the separating Herz-Schur multipliers.
Let $G$ be a locally compact $\sigma$-compact group (see Lemma
\ref{lem-secound-countable-sigma-compact}).
Suppose $1 \leq p < \infty$.
Let $\varphi \in L^\infty(G)$ and define $\phi \in L^\infty(G^2)$ by 
$\phi(s,t) = \varphi(s^{-1}t)$.
The Schur multiplier $T_\varphi^{HS} := T_\phi$ is called 
a Herz-Schur multiplier (with symbol $\varphi$).

In \cite[Proposition 4.5]{CrN13}, it is shown that a Herz-Schur 
multiplier $T_\varphi^{HS}\colon B(L^2(G)) \to B(L^2(G))$ with 
positive definite $\varphi$ such that $\varphi(e) = 1$, 
is a conjugation with a unitary if and only if 
$\varphi$ is a character.

\begin{corollary}
Let $1 \leq p < \infty$.
Let $G$ be a locally compact $\sigma$-compact group.
Let $T_\varphi^{HS}$ be a bounded Herz-Schur multiplier on $S^p(L^2(G))$.
Then $T_\varphi^{HS}$ is separating if and only if there exists a continuous character $\psi \colon G \to \Tdb$ and $c \in \Cdb$ such that $\varphi = c \psi$ almost everywhere.
\end{corollary}

\begin{proof}
If $\varphi(s) = c \psi(s)$ a.e. $s \in G$ for some continuous character $\psi$, then the symbol $\phi$ of the Schur multiplier satisfies $\phi(s,t) = \varphi(s^{-1}t) = c \psi(s^{-1}t) = c \psi^{-1}(s) \psi(t)$ for a.e. $(s,t) \in G^2$.
Since $\psi^{-1}$ and $\psi$ are clearly unitaries of $L^\infty(G)$,
by Theorem \ref{4Main}, $T_\varphi^{HS} = T_\phi$ is separating.

Conversely, assume $T_\varphi^{HS}$ separating.
Then by Theorem \ref{4Main}, there are unitaries $\alpha,\beta \in L^\infty(G)$ and some $c \in \Cdb$ such that $\varphi(s^{-1}t) = c \alpha(s) \beta(t)$ for a.e. $(s,t) \in G^2$.
Let $r \in G$.
Then $c \alpha(s) \beta(t) = \varphi(s^{-1}t) = \varphi((rs)^{-1}(rt)) = c \alpha(rs)\beta(rt)$ for a.e. $(s,t) \in G^2$.
Leaving the trivial case $c = 0$ aside, we deduce that
\begin{equation}
\label{equ-1-Herz-Schur}
\frac{\beta(rt)}{\beta(t)} = \frac{\alpha(s)}{\alpha(rs)}
\end{equation}
for a.e. $(s,t) \in G^2$.
Thus there is some $s \in G$ such that \eqref{equ-1-Herz-Schur} holds for a.e. $t \in G$.
Defining $\psi(r)$ as the right hand side of \eqref{equ-1-Herz-Schur}, we then obtain $\frac{\beta(rt)}{\beta(t)} = \psi(r)$ for a.e. $t \in G$.
The function $\psi \colon G \to \Cdb$ with this property is necessarily unique.
From $\psi(e) = 1$ and
\[\psi(r_1r_2) = \frac{\beta(r_1r_2t)}{\beta(t)} = \frac{\beta(r_1(r_2t))}{\beta(r_2t)} \frac{\beta(r_2t)}{\beta(t)} = \psi(r_1) \psi(r_2) \quad (\text{a.e.}\: t \in G)\]
we infer that $\psi$ is a character.
Since $\psi$ is measurable, by \cite[Corollary 22.19 p.~346]{HR}, it is automatically continuous.
From $\beta(rt) = \beta(t) \psi(r)$ for a.e. $t \in G$, we infer by a Fubini argument that there exists some $t \in G$ such that this equality holds for a.e. $r \in G$.
Thus, $\beta$ coincides a.e. with a continuous function.
Choosing this continuous representative for $\beta$, we obtain that for every $r$ in $G$, $\beta(rt)=\beta(t)\psi(r)$ for a.e. $t\in G$. 
Since $\beta$ is continuous, this implies
$\beta(rt)=\beta(t)\psi(r)$ for all $r,t$ in $G$. In particular, we have that $\beta(r)=\beta(e)\psi(r)$ for all $r\in G$.
%\[\frac{\beta(re)}{\beta(e)} = \lim_{V \to \{e\}} \frac{1}{\mu(V)} \int_V \frac{\beta(rt)}{\beta(t)} dt = \lim_{V \to \{e\}} \frac{1}{\mu(V)} \int_V \psi(r) dt = \psi(r) \quad (r \in G),\]
%where the limits are for neighborhoods $V$ of $e$ directed by inclusion.
%On the other hand, from \eqref{equ-1-Herz-Schur}, we deduce that there exists some $\tilde{\psi} \colon G \to \Cdb$ such that $\frac{\alpha(rs)}{\alpha(s)} = \tilde{\psi}(r)$ for a.e. $s \in G$.
Using \eqref{equ-1-Herz-Schur}, the same argument as above shows that $\alpha(s)=\alpha(e)\psi(s^{-1})$ for all $s$ in $G$.
%$\tilde{\psi}$ is a continuous character, $\alpha$ coincides a.e. with a continuous function and if we choose the continuous representative for $\alpha$, we get $\alpha = \alpha(e) \tilde{\psi}$.
%Moreover, necessarily $\tilde{\psi}= \psi^{-1}$.
Hence we deduce that $\varphi(s^{-1}t) = c \alpha(e) \beta(e) \psi(s^{-1}t)$ for a.e. $(s,t) \in G^2$. Therefore, $\varphi$ coincides a.e. with a multiple of a continuous character.
\end{proof}

\bigskip\noindent
{\bf Acknowledgements:}
C\'edric Arhancet and Christoph Kriegler 
were supported by the grant ANR-18-CE40-0021 of the French National 
Research Agency ANR (project HASCON).
Christoph Kriegler was supported by the grant ANR-17-CE40-0021 (project Front).
Christian Le Merdy was supported by the ANR project Noncommutative analysis on groups
and quantum groups (No./ANR-19-CE40-0002). 
Safoura Zadeh was supported by I-SITE Emergence project 
MultiStructure (Harmonic Analysis of Fourier and Schur multipliers) of Clermont Auvergne University.

\vskip 1cm

\bibliographystyle{abbrv}

\begin{thebibliography}{99}
\bibitem{Arh11}{\sc C. Arhancet}, {\it Noncommutative Fig\`a-Talamanca–Herz algebras for Schur multipliers},
Integral Equations Operator Theory 70 (2011), no. 4, 485–-510.
\bibitem{Arh23}{\sc C. Arhancet}, {\it Quantum information theory via Fourier multipliers on quantum groups}, Preprint 2020, arXiv:2008.12019.
\bibitem{AK} {\sc C. Arhancet and C. Kriegler}, {\it Projections, multipliers and decomposable maps on noncommutative Lp-spaces}, to appear in M\'emoires de la Soci\'et\'e Math\'ematique de France (2023), 177 (new series).
%\bibitem{BCR} {\sc C. Berg, J. P. R. Christensen and P. Ressel}, {\it Harmonic analysis on semigroups}, Theory of positive definite and related functions. Graduate Texts in Mathematics, 100. Springer-Verlag, New York, 1984.
\bibitem{CrN13} {\sc J. Crann and M. Neufang}, {\it Quantum channels arising from abstract harmonic analysis}, J. Phys. A 46 (2013), no. 4, 045308, 22 pp.
\bibitem{Daws} {\sc M. Daws}, {\it Representing multipliers of the Fourier algebra on non-commutative $L^p$ spaces}, Canad. J. Math. 63 (2011), no. 4, 798--825.
\bibitem{DCH} {\sc J.~De Canni\`ere and U.~Haagerup}, {\it Multipliers of the Fourier algebras of some simple Lie groups and their discrete subgroups}, Amer. J. Math. 107 (1985), 455--500. 
%\bibitem{Dix} {\sc J. Dixmier}, {\it Von Neumann algebras}, North-Holland Mathematical Library, 27. North-Holland Publishing Co., Amsterdam-New York, 1981. xxxviii+437 pp. 
%\bibitem{Edw}{\sc R. E. Edwards}, {\it Functional analysis}, Theory and applications. Corrected reprint of the 1965 original. Dover Publications, Inc., New York, 1995. xvi+783 pp.
\bibitem{Eym} {\sc P. Eymard}, {\it L'alg\`ebre de Fourier d'un groupe localement compact (French)},
Bull. Soc. Math. France 92 (1964), 181--236.
\bibitem{F} {\sc G. Folland}, {\it A course in abstract harmonic analysis}, Textbooks in Mathematics. Boca Raton, 2016.
\bibitem{GJP}  {\sc A. González-Pérez and M. Junge and J. Parcet},
{\it Smooth Fourier multipliers in group algebras via Sobolev dimension}, 
Ann. Sci. Éc. Norm. Supér. (4) 50 (2017), no. 4, 879–925.
\bibitem{Haa} {\sc U. Haagerup}, {\it Decomposition of completely bounded maps on operator algebras}, Unpublished, Odense
University, Denmark, 1980.
\bibitem{HOS} {\sc H. Hanche-Olsen and E. Størmer}, {\it Jordan operator algebras}, Monographs and Studies in Mathematics Series Profile, 21. Boston - London - Melbourne: Pitman Advanced Publishing Program. VIII, 1984.
\bibitem{Herz} {\sc C. Herz}, {\it Une g\'en\'eralisation de la transform\'ee de
Fourier-Stieltjes},  Ann. Inst. Fourier 24 (1974), no. 3, 
145–157. 
\bibitem{HR} {\sc E. Hewitt and K. A. Ross}, {\it Abstract harmonic analysis}, Vol. I Second edition, Grundlehren der Mathematischen Wissenschaften, 115. Springer-Verlag, Berlin-New York, 1979.
\bibitem{HRW} {\sc G. Hong, S. K. Ray and S. Wang}, 
{\it Maximal ergodic inequalities for positive operators on
noncommutative Lp-spaces}, Preprint 2019, arXiv:1907.12967
\bibitem{JRS}{\sc M. Junge and Z. Ruan and D. Sherman}, {\it A classification for 2-isometries of noncommutative $L^p$-spaces}, Isr. J. Math. 150 (2005), 285-314.
\bibitem{JMP1} {\sc M. Junge and T. Mei and J. Parcet}, {\it Smooth Fourier multipliers on group von Neumann algebras}, Geom. Funct. Anal. 24 (2014), no. 6, 1913–1980.
\bibitem{JMP2} {\sc M. Junge and T. Mei and J. Parcet}, {\it Noncommutative Riesz transforms—dimension free bounds and Fourier multipliers}, J. Eur. Math. Soc. (JEMS) 20 (2018), no. 3, 529–595.
\bibitem{KR} {\sc R. V. Kadison and J.R. Ringrose}, 
{\it Fundamentals of the theory of operator algebras, Volume II}, Graduate Studies in Mathematics
16, Amer. Math. Soc. 1997.
\bibitem{KaL18} {\sc E. Kaniuth and A. T.-M. Lau}, {\it Fourier and Fourier-Stieltjes algebras on locally compact groups}, Mathematical Surveys and Monographs, 231. American Mathematical Society, Providence, RI, 2018. xi+306 pp.
\bibitem{LD} {\sc V. Lafforgue and M. de la salle}, {\it Noncommutative $L^p$-spaces without the completely bounded approximation property}, Duke Math. J. 160 (2011), No. 1, 71-116.
\bibitem{Larsen} {\sc R. Larsen}, 
{\it An introduction to the theory of multipliers}, 
Springer-Verlag, New York-Heidelberg, 1971. xxi+282 pp.
\bibitem{LZ} {\sc C. Le Merdy and S. Zadeh}, {\it $\ell^1$-contractive maps on noncommutative $L^p$-spaces}, Journal of Operator Theory, 85:2(2021), 417–442.
\bibitem{LZ2} {\sc C. Le Merdy and S. Zadeh}, {\it On factorization of separating maps on noncommutative Lp-spaces}, Indiana University Mathematics Journal, 71  (2022), no. 5, 1967–2000
\bibitem{LZ3} {\sc C. Le Merdy and S. Zadeh}, {\it Surjective separating maps on noncommutative $L^p$-spaces}, Mathematische Nachrichten, Math. Nachr. 295 (2022), no. 1, 175–188.
\bibitem{MR} {\sc T. Mei and \'E. Ricard}, {\it Free Hilbert transforms}, Duke Math. J. 166 (2017), No. 11, 2153-2182.
\bibitem{PRD} {\sc J. Parcet and \'E. Ricard and M. de la salle}, {\it Fourier multipliers in $SL_n(\Rdb)$}, Duke Math. J. 171 (2022), No. 6, 1235-1297.
\bibitem{Par} {\sc S. Parrott}, {\it Isometric multipliers}, Pac. J. Math. 25 (1968), 159-166.
\bibitem{Pau02} {\sc V. Paulsen}, {\it Completely bounded maps and operator algebras}, Cambridge Univ. Press, 2002.
\bibitem{PX} {\sc G. Pisier and Q. Xu}, {\it Non-commutative $L^p$-spaces}, Handbook of the geometry of Banach spaces, Vol. 2, pp. 1459-1517, North-Holland, Amsterdam, 2003.
\bibitem{RS} {\sc M. Reed and B. Simon}, 
{\it Methods of modern mathematical physics. I. Functional analysis, Second edition},
Academic Press, New York, 1980, xv+400 pp.
\bibitem{Spronk} {\sc N. Spronk},
{\it Measurable Schur multipliers and completely bounded multipliers of the Fourier algebras} 
Proc. London Math. Soc. (3) 89 (2004), no. 1, 161-192.
\bibitem{S} {\sc E. St\o rmer}, {\it On the Jordan structure of $C^*$-algebras}, 
Trans. Amer. Math. Soc. 120 (1965), 438-447.
\bibitem{St} R. Strichartz, {\it Isomorphisms of group  algebras}, Proc. Am. Math. Soc. 17 (1966), 858-862.
\bibitem{Tak} {\sc M. Takesaki}, {\it Theory of operator algebras. I.}, 
Springer-Verlag, New York-Heidelberg, 1979. 
\bibitem{Tak2} {\sc M. Takesaki}, {\it Theory of operator algebras. II.},
Encyclopaedia of Mathematical Sciences, 125. Operator Algebras and Noncommutative Geometry, 6. Springer-Verlag, Berlin, 2003. 
\bibitem{Yeadon} {\sc F. Y. Yeadon}, {\it Isometries of noncommutative $L^{p}$-spaces}, 
Math. Proc. Cambridge Philos. Soc. 90 (1981),  41-50.
\end{thebibliography}

\end{document}